\documentclass[conference]{IEEEtran}

\usepackage{graphicx}
\graphicspath{ {images/} }

\usepackage{multicol,lipsum}
\usepackage {tikz}
\usetikzlibrary {positioning}
\usepackage{mathtools}
\usepackage{caption}
\usepackage[ruled,vlined,linesnumbered]{algorithm2e}
\usepackage[export]{adjustbox}
\usepackage{breqn}
\usepackage{amsmath,amsthm,amsfonts}
\usepackage{epsfig}
\usepackage{subcaption}
\usepackage{url}
\usepackage{balance}
\usepackage{array}
\usepackage{xcolor}
\SetArgSty{textup}

\SetCommentSty{mycommfont}

\newcolumntype{L}[1]{>{\raggedright\let\newline\\\arraybackslash\hspace{0pt}}m{#1}}
\newcolumntype{C}[1]{>{\centering\let\newline\\\arraybackslash\hspace{0pt}}m{#1}}
\newcolumntype{R}[1]{>{\raggedleft\let\newline\\\arraybackslash\hspace{0pt}}m{#1}}

\newtheorem{theorem}{Theorem}[section]

\newtheorem{lemma}[theorem]{Lemma}

{\theoremstyle{remark} }
{\theoremstyle{notation} }
{\theoremstyle{definition}

}

\newcommand{\bg}[1]{\medskip\noindent{\bf #1}}

\newcommand{\hide}[1]{}

\IEEEoverridecommandlockouts
\begin{document}
	
	\title{Dynamic Pricing in Smart Grids under Thresholding Policies: Algorithms and Heuristics}
	
	\author{\IEEEauthorblockN{Zaid Almahmoud, Jacob Crandall,   Khaled Elbassioni, Trung Thanh Nguyen, \& Mardavij Roozbehani    }
		
		\thanks{Z. Almahmoud and K. Elbassioni are with the department of Electrical Engineering and Computer Science at Masdar Institute of Science and Technology, Abu Dhabi, United Arab Emirates	(e-mail: \{zjalmahmoud,kelbassioni\}@masdar.ac.ae).  
			J. Crandall is with the department of Computer Science at Brigham Young University, Provo, UT (e-mail: crandall@cs.byu.edu).
			T. Nguyen is with is with the department of Computer Science at
			New York University, Abu Dhabi, United Arab Emirates (e-mail: ttn1@nyu.edu). 
			M. Roozbehani is with the Laboratory for Information and Decision Systems, Massachusetts Institute of Technology (e-mail: mardavij@mit.edu).
		}
	}

	\maketitle

	\begin{abstract}
		
		Minimizing the peak power consumption and matching demand to supply, under fixed threshold polices, are two key requirements for the success of the future electricity market. In this work, we consider dynamic pricing methods to minimize the peak load and match demand to supply in the smart grid. As these optimization problems are computationally hard to solve in general, we propose generic heuristics for approximating their solutions. Further, we provide theoretical analysis of uniform pricing in peak-demand minimization. Moreover, we propose optimal-pricing algorithms for scenarios in which the time-period in which tasks must be executed is relatively small. Finally, we conduct several experiments to evaluate the various algorithms on real data.
		
	\end{abstract}

	\section{Introduction}
	The International Energy Agency reported in 2009 that power consumption has dramatically increased over the last 25 years \cite{al2014reducing}.  Importantly, buildings account for about 40\% of the total energy consumption \cite{hammad2010energy}. 
	Consequently, smart and sustainable solutions need to optimize power consumption and minimize electricity loads.

	Smart grids combine advanced monitoring and communication technologies to provide energy in a smart, efficient, and user friendly manner \cite{logenthiran2012demand}.
	With a smart grid, energy providers could use dynamic pricing to optimize objectives, such as minimizing the peak demand and matching time-varying demand to time-varying supply~\cite{borgs2014optimal}.  In dynamic pricing, the price of electricity is varied over time to encourage consumers to alter their behavior. However, to do so effectively, energy providers must understand consumer behavior, including their demands and flexibility~\cite{roozbehani2012volatility}, since strategic consumers are likely to shift consumption to periods with low prices~\cite{ohannessian2014dynamic}. When energy providers know or can effectively estimate the algorithms (and parameters) used by consumers to make decisions, they can potentially set prices to effectively influence collective consumer behavior.

	In this work, we consider two optimization problems in the smart grid: minimizing peak demand and matching time-varying demand to time-varying supply.  In both problems, we consider a composition of consumers in a power system that use a threshold policy to determine when to consume electricity. The composition includes the number of arrivals, their demands, and their consumption deadlines, at different time periods.  We assume that the consumer has a single amount of demand to be consumed at a time period between her arrival and deadline.  For both problems, we consider scenarios in which the consumers' composition is both known and unknown beforehand.
	
	This model gives rise to optimization problems that are computationally (NP-)hard to solve in general. To overcome this intrinsic difficulty, we first propose and evaluate two generic heuristics for providing an approximate solution for the two problems.  While we are not able to theoretically analyze  in general how close to optimal the solutions produced by these heuristics, we provide theoretical analysis of the approximation gap in the case when {\it uniform pricing} is applied to minimize peak demand. Moreover, we propose {\it optimal} pricing algorithms that can be used when the maximum deadline period of the power jobs is relatively small. Finally, we conduct several experiments to evaluate the performance of the proposed algorithms, which are validated by the use of real appliances data available on Pecan Street's Dataport website \cite{dataport}.  We begin by reviewing related prior work.


\section{Literature Review}
\label{review}

\subsection{Estimation of Price-Response}
To effectively set prices, it is important to estimate the consumers' response to the chosen prices. Several models have been considered to estimate the consumers' demand as a function of price \cite{conejo2010real}, \cite{ohannessian2014dynamic}, \cite{ahn2007pricing}. We focus on the model presented by M. Ohannessian et al. \cite{ohannessian2014dynamic}, as we use it in this work. In this model, consumers have demands that must be met be a certain deadline.  A consumer may delay her consumption up until this deadline if prices are not satisfactory.  Consumers use a threshold policy to determine when to consume, in which they consume when the prices falls below the computed threshold.

Given this problem formulation, the authors presented a formula for computing the aggregate consumption as a function of the price when the arrivals and their demands are known beforehand. The authors additionally presented an averaged model that can be used when the exact knowledge about the number of arrivals and their demands is unknown.  The model assumes knowledge about the arrival and demand rates, and uses them to estimate the expected consumption.  When the consumers' composition is unknown beforehand, the authors proposed an estimator that uses the history of the price and the total consumption to estimate the arrival and demand rates. The authors showed that, when the number of time periods in the history is sufficiently large, the proposed estimator is consistent and unbiased.

\subsection{Peak-Demand Pricing}
C. Ibars, M. Navarro, and L. Giupponi \cite{ibars2010distributed}, proposed a distributed load-management technique for controlling the consumer load profile using dynamic pricing. The technique aims balance consumption across the different hours of the day. The problem was formulated as a congestion game that models a competition over the network, where the cost is a function of the congestion level. The game was shown to converge to a Nash equilibrium, after a finite number of improvement steps.  Simulations showed that the proposed demand-management scheme leads to a reduction in the peak demand, compared to the unmanaged demand scheme.

N. Li, L. Chen, and S.H. Low \cite{li2011optimal} showed that dynamic pricing can be used to manage the consumers' demand to benefit both consumers and the utility.  The objective of the utility company, which serves a group of consumers, is to maximize the social welfare, and the objective of the individual consumer is to maximize her net benefit. The authors prove that dynamic pricing can lead to the optimality of the social welfare as well as the individual objective, reaching a Nash equilibrium.  Based on this result, the authors proposed a distributed algorithm in which the consumers and the utility company compute the equilibrium iteratively. Simulations showed that the proposed scheme balances demand over time, and, thus, effectively reduces the peak load.

\subsection{Matching Demand to Supply}
S.D. Ramchurn et al. \cite{ramchurn2012putting} argued that balancing demand and supply is achieved in today's power grid through a real-time varying of the supply to match the demand. Nevertheless, the authors suggested that a more powerful balancing scheme is where the demand follows the supply, as its flattens the peak load, avoids overloading the generators, and leads to fast recovery given power failures.  These authors proposed time-of-use (TOU) pricing, which provides an expensive price during peak hours. However, it has been shown that such pricing methods result in a high peak demand during the off-peak hours \cite{strbac2008demand}. This drawback of static pricing has led many researchers to consider dynamic pricing \cite{roozbehani2012volatility}.

T.K. Wijaya, K.M. Larson, and K. Aberer \cite{wijaya2013matching} proposed a methodology in which the peak-to-average-ratio (PAR) is explicitly cut from the supply load through a PAR-Cut algorithm, and the consumers adapt to the resulting load. The adaptation is done through a multiunit auction, which results in a redistribution of the load. The auction provides truthful bidding--a consumer has no incentive to lie about her valuation. The experimental results demonstrated a significant cost reduction when the cut percentage is above 20\%. Moreover, the experiments showed that the consumers can save up to approximately 20\% of their electricity bill, depending on the PAR cut percentage and the consumers' valuation. Finally, the company can gain up to about 10\% in additional revenue depending on the PAR cut.

J. A. Taylor et al. \cite{taylor2013consolidated} considered consolidated dynamic pricing. The regulation includes services that adapt in real-time to unexpected system changes.  The authors proposed a regulation scheme based on optimal control-based pricing \cite{berger1989real}, while incentivizing a linear quadratic optimal regulation to reach a competitive equilibrium. Finally, the authors utilized tools from mechanism design and convex optimization to make the approach significantly more adaptable and practical.

\subsection{Our Contribution}
In our work, we utilize a simplified version of the model presented by M. Ohannessian et al. \cite{ohannessian2014dynamic} to estimate consumer price response to solve the problems of peak-demand minimization and matching demand to supply using dynamic pricing.  We also utilize the estimator proposed by M. Ohannessian et al.~\cite{ohannessian2014dynamic} to estimate the consumers' composition when it is unknown.

\section{System Model}
\label{model}

\subsection{Problem Formulation - peak-demand minimization}
In this section, we present our mathematical formulation for the problem of peak-demand minimization, as follows. We denote the time by $k\in \{1, \dots, K \}$. Let $A_{n}(k)$ denote the number of arrivals at time $k$, with deadline $n \in \{1, \dots, N\}$. Without loss of generality, we assume that each consumer has a single amount of demand to be consumed at a time period between her arrival and deadline. Therefore, we ignore the backlog demand notation used by M. Ohannessian et al. \cite{ohannessian2014dynamic}. Consequently, the demand of consumer $j$ is independent of time, and denoted by $d_{j}$. Let $\lambda(k)$ denote the price per unit of consumption at time $k$. Let $\tau_{t}$ denote the threshold policy of consumers who have time to go $t$, where $t \in \{1, \dots, N \}$. Let $G_{s,t}(k)$ denote the group of consumers at time $k$, who have time to go $t$, have been in the system for $s$ time periods, where $s \in \{0, \dots, N-1\}$, and have not yet consumed their demands. Let $u(k,\lambda)$ denote the total power consumption at time $k$ when the price is $\lambda$. Our objective is to set the prices at the different time periods such that the peak demand is minimized. More formally, the problem formulation is,
{
\begin{align}
\label{pdp_o}
&\min \quad \{\max\limits_{1\leq k \leq K}u(k,\lambda)\}\\
\text{s.t.}&\quad
\label{pdp_c1}
\forall{k=1, \dots, K}: u(k,\lambda) =\sum_{t:\lambda \leq \tau_t}^{} \sum_{s=0}^{N-t} \left ( \sum_{j \in G_{s,t}(k)} d_{j}\right)
\end{align}
}

 Constraint (\ref{pdp_c1}) ensures that the consumption at time $k$ accounts for all consumers with threshold policies that accept the chosen price at time $k$.

\subsection{Problem Formulation - Matching Demand To Supply}

In this section, we present our mathematical formulation for the problem of matching demand to supply. We will use the formulation presented in the previous section, with a simple modification as follows. Let $S(k)$ denote the power supply at time $k$. Recall that $u(k,\lambda)$ denotes the total power consumption at time $k$ when the price is $\lambda$. Our objective is to set the price such that the MSE between the total power consumption and the power supply is minimized. More formally, the problem formulation is,

\begin{equation}
\label{obj_dsm}
\min \{\frac{1}{K}\sum_{k=1}^{K}(u(k,\lambda)-S(k))^{2}\}
\end{equation}
such that (\ref{pdp_c1}) is satisfied.

It can be shown that this problem is NP-complete\footnote{This means that this problem belongs to a class of problems for which a polynomial time (efficient) algorithm is unlikely to exist.} (see the appendix). 
Thus, in Section~\ref{generic} we will consider heuristics for solving the problem.

\subsection{Assumptions}
We assume that the consumers' arrival follows a Poisson process, which is independent for the different deadlines and the demand process, and independent and identically distributed (i.i.d.) across time with a mean $\alpha_{n}$, where $n \in \{1, \dots, N \}$. This assumption is consistent with the system model presented by M. Ohannessian et al.~\cite{ohannessian2014dynamic}.  We also assume that the maximum deadline period $N$ is constant. Finally, we assume that the threshold policy $\tau_{t}$ is monotonically increasing with the decrease of $t$, and that all threshold policies guarantee consumption by the consumers' deadlines.

\label{np}



\section{Heuristics}
\label{generic}
\subsection{Greedy Heuristic}
In this section, we describe a generic greedy heuristic that can handle several pricing optimization problems, including the problems of peak-demand minimization, and matching demand to supply. In addition to the heuristic description, we provide an analysis for the time complexity of the heuristic.

\subsubsection{Algorithm Description} 
The pseudocode of the Greedy heuristic is presented in Algorithm \ref{greedy_heuristic}. At each time period, the heuristic chooses a price among the threshold policies, such that the optimization objective is locally achieved. Achieving the optimization objective locally is defined as optimizing the consumption at time period $k$, without considering the other time periods. The chosen prices at the different time periods are finally returned as an output. The Greedy heuristic can be used to solve the peak-demand minimization problem, by setting the optimization objective to (\ref{pdp_o}). Similarly, the problem of matching demand to supply can be solved using the Greedy heuristic by setting the optimization objective to (\ref{obj_dsm}).

\begin{algorithm}

\DontPrintSemicolon

 \caption{\sc {Greedy Heuristic}}
 \label{greedy_heuristic}

  \SetKwInOut{Input}{input}
  \SetKwInOut{Output}{output}
  
  \BlankLine
 \Input{Optimization objective $OB$, time horizon $K$, threshold policies $T=\{\tau_{t}\}_t$} 
 \BlankLine

 \Output{Optimal prices $\lambda$ that achieve $OB$, satisfying (\ref{pdp_c1}) }
 
  \BlankLine

\For{k=1 to $K$}
{
 \BlankLine

$\lambda$[$k$]= Choose a price $\tau_{t} \in T$, that achieves $OB$ locally at time period $k$

}
  \BlankLine

return $\lambda$
\end{algorithm}

\begin{theorem}
The expected runtime of the Greedy heuristic is $O(K N \sum_{t=1}^{N}\sum_{s=0}^{N-t} \alpha_{s+t})$.
\end{theorem}
\begin{proof}
The Greedy heuristic performs $K$ iterations to set the prices at the different time periods. In each iteration, the heuristic loops over $N$ threshold policies to select the optimal one. For each threshold policy, the heuristic accumulates the consumers' demand that would result when setting the price to that threshold policy. To identify the number of consumers whose demands to be accumulated at each time period, we recall the notation $G_{s,t}(k)$, which denotes the group of consumers at time $k$, who have time to go $t$, and have been in the system for $s$ time periods. Based on our early assumptions, it is obvious that the expected number of consumers in $G_{s,t}(k)$ is $\alpha_{s+t}$. Therefore, the expected number of consumers at each time period is no more than $\sum_{t=1}^{N}\sum_{s=0}^{N-t} \alpha_{s+t}$. Consequently, the expected runtime of the Greedy heuristic is $O(K N \sum_{t=1}^{N}\sum_{s=0}^{N-t} \alpha_{s+t})$.
\end{proof}

\subsection{Sliding-Window Heuristic}
In this section, we describe another generic heuristic called Sliding-Window, which can be used to solve the two formulated problems. As in the previous section, we additionally analyze the time complexity of the Sliding-Window heuristic. Further, we discuss the potential improvement in the approximation ratio of the Sliding-Window over the Greedy heuristic.

\subsubsection{Algorithm Description} 
The pseudocode of the Sliding-Window heuristic is presented in Algorithm \ref{sliding-window_heuristic}. In addition to the optimization objective, the heuristic receives a constant window size $W$ as an input. At each time period $k$, the heuristic finds the locally optimal sequence of prices in the time interval [$k$, $k+W-1$], by brute force. The price found at time $k$ is then stored in the final solution. The Sliding-Window heuristic can be used to solve the peak-demand minimization problem, by simply setting the optimization objective to (\ref{pdp_o}). Similarly, the problem of matching demand to supply can be solved using the Sliding-Window heuristic by setting the optimization objective to (\ref{obj_dsm}).

\begin{algorithm}

\DontPrintSemicolon

 \caption{\sc {Sliding-Window Heuristic}}
 \label{sliding-window_heuristic}

  \SetKwInOut{Input}{input}
  \SetKwInOut{Output}{output}
  
  \BlankLine
 \Input{Optimization objective $OB$, window size $W$, time horizon $K$, threshold policies $T=\{\tau_{t}\}_t$} 
 \BlankLine

 \Output{Optimal prices $\lambda$ that that achieve $OB$, satisfying (\ref{pdp_c1}) }
 
  \BlankLine

$first=1$; $last=W$\;
\While{$last \leq K$}
{

\BlankLine

$\lambda$[$first$], $\lambda$[$first+1$], $\dots$, $\lambda$[$last$] = Choose a sequence of prices $S=(\tau_{a}, \tau_{b}, \dots, \tau_{z}$), where $|S|=W$, and $\tau_{a},\tau_{b},\dots,\tau_{z} \in T$, that achieves $OB$ locally in the time interval [$first$, $last$], by brute force.
\BlankLine

$first$++;
$last$++\;
}
  \BlankLine

return $\lambda$
\end{algorithm}

\begin{theorem}
The expected runtime of the Sliding-Window heuristic is $O(K N^{W} \sum_{t=1}^{N}\sum_{s=0}^{N-t} \alpha_{s+t})$.
\end{theorem}

\begin{proof}
The Sliding-Window heuristic performs $K-W+1$ iterations to set the prices at the different time periods. At each time period $k$, where $k \leq K-W+1$, the heuristic finds the optimal sequence of prices in the window between the time periods $k$, and $k+W-1$, by brute force. This is done by iterating over all the possible sequence of prices in the window, which are in total $N^{W}$ sequences. In addition, for each sequence of prices, the heuristic accumulates the demand response to these prices in the window which is of size $W$. The expected number of operations to accumulate the demands at each time period is no more than $\sum_{t=1}^{N}\sum_{s=0}^{N-t} \alpha_{s+t}$, as mentioned earlier. Therefore, the expected runtime of the heuristic is $O(K N^{W} \sum_{t=1}^{N}\sum_{s=0}^{N-t} \alpha_{s+t})$.
\end{proof}


\subsubsection{Improvement Over the Greedy}
The look-ahead feature of the Sliding-Window while setting the price at each time period, is expected to provide an improvement in terms of the approximation ratio over the Greedy heuristic, which does not take the future time periods into account. The improvement in the approximation ratio is subject to the window size $W$. The increase in the window size should naturally reduce the approximation ratio of the heuristic. This is due to the fact that larger window size means larger number of look-ahead time periods to be considered, leading to a larger optimization coverage. However, increasing the window size comes with the drawback of reducing the heuristic efficiency in terms of runtime, since the work done on the window increases exponentially with the increase of the window size. Therefore, the choice of the window size should consider a trade-off between the efficiency and the approximation to the optimal solution. Finally, it is worth mentioning that the Greedy heuristic is a special case of the Sliding-Window heuristic, where $W$=1. Also, when $W=K$, the Sliding-Window heuristic is obviously optimal.

\subsection{Uniform Pricing Approach - Theoretical Analysis}
In this section, we study the uniform pricing approach for solving the problem of peak-demand minimization. The approach fixes the price $\lambda$ to a threshold policy $\tau_{t}$ during all time periods. We provide theoretical analysis to show that in the typical configuration of the system, where $K$ is fixed, the approximation ratio of the uniform pricing approach on the average case is constant. We note that our analysis assumes homogeneous demands, in the sense that all consumers in the system have the same amount of demand, denoted by $d$.

\begin{lemma}
\label{l1}
The peak demand of uniform pricing is no more than 
$d \sum_{n=1}^{N}\max\limits_{1\le k \le K}\{A_{n}(k)\}$ if demand is homogeneous.
\end{lemma}

\begin{proof}
At each time period $k$, where $N\le k \le K-N+1$, the algorithm sets the price to a threshold policy $\tau_{t}$, causing the consumers in $g_{1}(k)$ to consume their demands, and the consumers in $g_{2}(k)$, to delay their consumption, as mentioned earlier. This accumulates a total consumption of $d (A_{1}(k)+A_{2}(k)+\dots+A_{t}(k))$, given that the demand is $d$ for all consumers. In addition, there are other groups of consumers who could be present at time period $k$, proceeding from the previous time periods. We note that among these consumers, no consumer is present with time to go $v$, where $v < t$, since she would have consumed her demand at an earlier time period, when her time to go was equal to $t$. Also, among the consumers proceeding from the previous time periods, the consumers with time to go $w$, where $w > t$, will obviously delay their consumption, since $\tau_{w} >\tau_{t}$. We are left with the groups of consumers who have time to go $t$, proceeding from the previous time periods. We denote these groups by a third set, namely $g_{3}(k)=\{G_{1,t}(k),G_{2,t}(k),\dots, G_{N-1,t}(k)\}$. Since the threshold policy of these groups is $\tau_{t}$, the consumers in these groups will consume their demands at time period $k$. This accumulates an additional consumption of $d(A_{t+1}(k-1)+A_{t+2}(k-2)+\dots+A_{N}(k-N+t))$. Hence, the total consumption $u(k)$, where $N\le k \le K-N+1$, is $d(\sum_{n=1}^{t} A_{n}(k) + \sum_{n=t+1}^{N} A_{n}(k+t-n))$. Note that for all other time periods $m$, where $m < N$, or $m > K-N+1$, the number of groups with time to go $t$, proceeding from the previous time periods, are obviously less than those at time period $k$. Therefore, neither $u(k)$, nor $u(m)$ can exceed the term $d\sum_{n=1}^{N}\max\limits_{1\le k \le K}\{A_{n}(k)\}$.

\end{proof}


\begin{lemma}
\label{l2}
The peak demand of the optimal algorithm is no less than
$d \sum_{n=1}^{N}\min\limits_{1\le k \le K}\{A_{n}(k)\}$ if demand is homogeneous.
\end{lemma}

\begin{proof}
We prove this lemma by contradiction as follows. Assume that the peak demand is less than $d \sum_{n=1}^{N}\min\limits_{1\le k \le K}\{A_{n}(k)\}$. Then, at time period $N$, there is a set of groups of consumers, namely  $g_{4}=\{G_{N-1,1}(N),G_{N-2,1}(N),\dots,G_{0,1}(N)\}$, who will not consume their demands by their deadline, which is the current time period $N$. This is true, since if they ever consumed their demands, then the peak demand should be at least $d \sum_{n=1}^{N}\min\limits_{1\le k \le K}\{A_{n}(k)\}$, given that the threshold policy is monotonically increasing with time. To prove this argument, we know that the group $G_{N-1,1}(N)$ have arrived at time period 1, with a total demand of $d(A_{N}(1))$. If this group have consumed their demands at time period 1, then $u(1)=d \sum_{n=1}^{N} A_n(1) \ge d \sum_{n=1}^{N}\min\limits_{1\le k \le K}\{A_{n}(k)\}$. Therefore, the group $G_{N-1,1}(N)$ would proceed to time period 2, when the group $G_{N-2,1}(N)$ arrives with a total demand of $d (A_{N-1}(2))$. Since both $G_{N-1,1}(N)$, and $G_{N-2,1}(N)$ have time to go $N-1$ at time period 2, if any of them consumed at time period 2, then $u(2)\ge (d \sum_{n=1}^{N-1} A_n(2))+d(A_N(1)) \ge d \sum_{n=1}^{N}\min\limits_{1\le k \le K}\{A_{n}(k)\}$. One can obviously proceed with this argument until time period $N$, where none of the groups in $g_{4}$ would have or will consume their demands. Having groups of consumers not consuming their demands by their deadline, contradicts with our early assumption that the threshold policies guarantee the consumption of all consumers by their deadline.
\end{proof}


Next, we provide an average case analysis for the approximation ratio of the uniform pricing approach.
More specifically, we are interested in computing the ratio of the expected upper bound provided in lemma \ref{l1}, over the expected lower bound provided in lemma 
\ref{l2}. More formally, we aim to study the value of the following term,

\begin{align*}
R=\frac{E[d\sum_{n=1}^{N}\max\limits_{1\le k \le K}\{A_{n}(k)\}]}{E[d\sum_{n=1}^{N}\min\limits_{1\le k \le K}\{A_{n}(k)\}]}
\end{align*}

\begin{flushleft}
Let $A_{max}^n=\max\limits_{1\le k \le K}\{A_{n}(k)\}$, and  $A_{min}^n=\min\limits_{1\le k \le K}\{A_{n}(k)\}$, we consider the following simplification for $R$, 
\end{flushleft}

\begin{align*}
R = \frac{d\sum_{n=1}^{N}E[A_{max}^n]}{d\sum_{n=1}^{N}E[A_{min}^n]} \le \max\limits_{1\le n \le N}\left\{ \frac{E[A_{max}^n]}{E[A_{min}^n]}\right\}
\end{align*}

\begin{flushleft}
Consequently, we obtain the following theorem, 
\end{flushleft}

\begin{theorem}
The ratio of the expected peak demand of the uniform pricing approach over the expected peak demand of the optimal algorithm is no more than $\max\limits_{1\le n \le N}\{ \frac{E[A_{max}^n]}{E[A_{min}^n]}\}$, when the demand is homogeneous.
\end{theorem}

\begin{flushleft}
According to \cite{mathMax} and \cite{mathMin},
\end{flushleft}
\begin{align*}
E[A_{max}^n] &= \sum^\infty_{x = 0}\left[1-\left(\sum_{i=0}^x  e^{-\alpha_{n}}\frac{(\alpha_{n})^{i}}{i!}\right)^K\right]\\
E[A_{min}^n] &= \sum_{x=1}^\infty \left(\sum_{i=x}^\infty e^{-\alpha_{n}}\frac{(\alpha_{n})^i}{i!}\right)^K
\end{align*}

Next, we provide empirical analysis for the values of $E[A_{max}^n]$ and $E[A_{min}^n]$, to show that in the typical configuration of the system, where $K$ is fixed, the ratio $R$ is constant. Fig. \ref{AmaxAmin} illustrates the values of $E[A_{max}^n]$ and $E[A_{min}^n]$ for various values of $\alpha_{n}$, and a fixed value of $K$, where $K=100$. Fig. \ref{R} illustrates the value of the ratio ($E[A_{max}^n]/E[A_{min}^n]$) for various values of $\alpha_{n}$, when $K=100$. It can be observed in Fig. \ref{R} that the ratio is a constant less than 2, and decreases with the increase of $\alpha_{n}$. Further, we observe that the decrease in the ratio becomes slower with the increase of $\alpha_{n}$. Therefore, we conclude that in the typical configuration of the system, the ratio $R$ is constant. Finally, it is worth mentioning that the Greedy heuristic described earlier is a special case of the uniform pricing approach, where $\lambda=\tau_{1}$, and hence would have the same theoretical guarantee as above.

%
%
%

\begin{figure}[t]
	\begin{flushleft}
		\begin{subfigure}[t]{0.25\textwidth}
			\centering
			\includegraphics[height=1.4in, width=\linewidth]{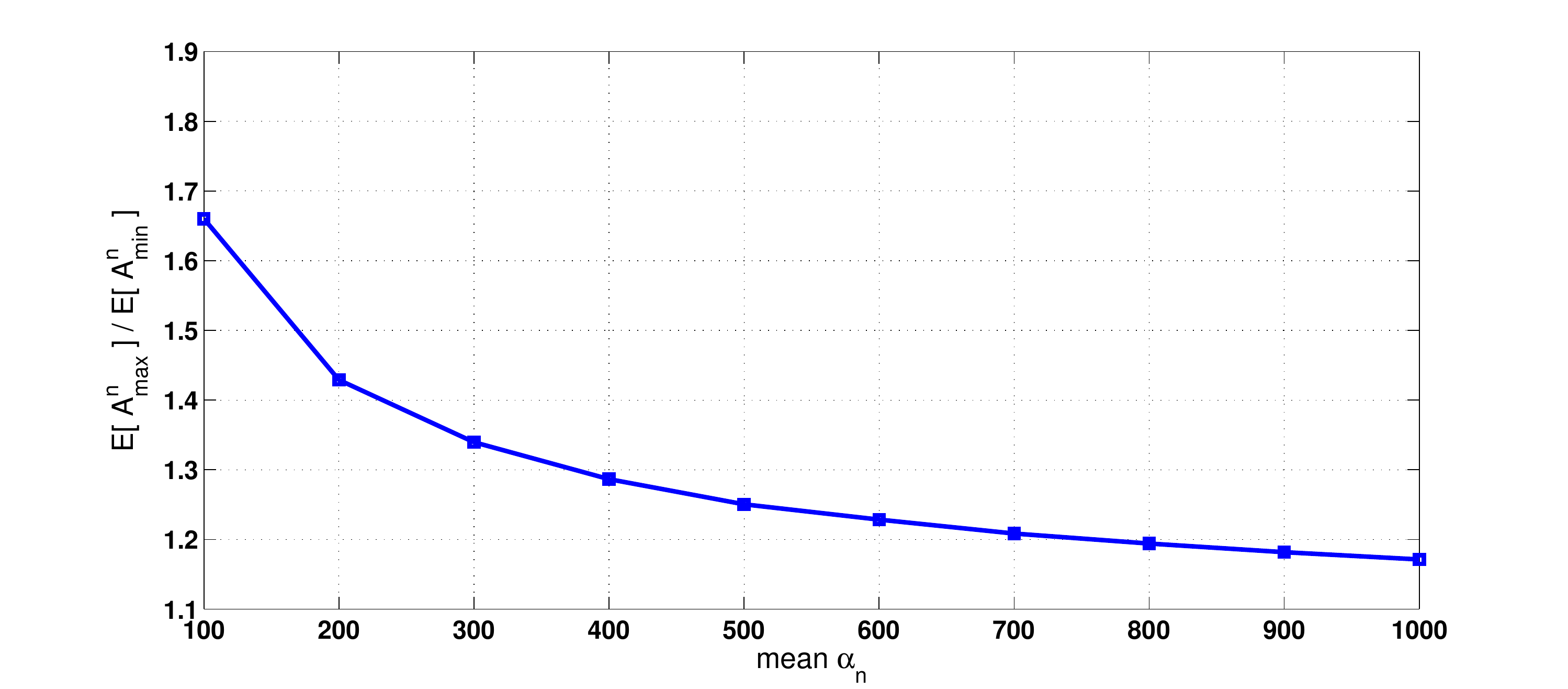}
			\caption{}
            \label{AmaxAmin}
		\end{subfigure}%
		~ 
		\begin{subfigure}[t]{0.25\textwidth}
			\centering
			\includegraphics[height=1.4in, width=\linewidth]{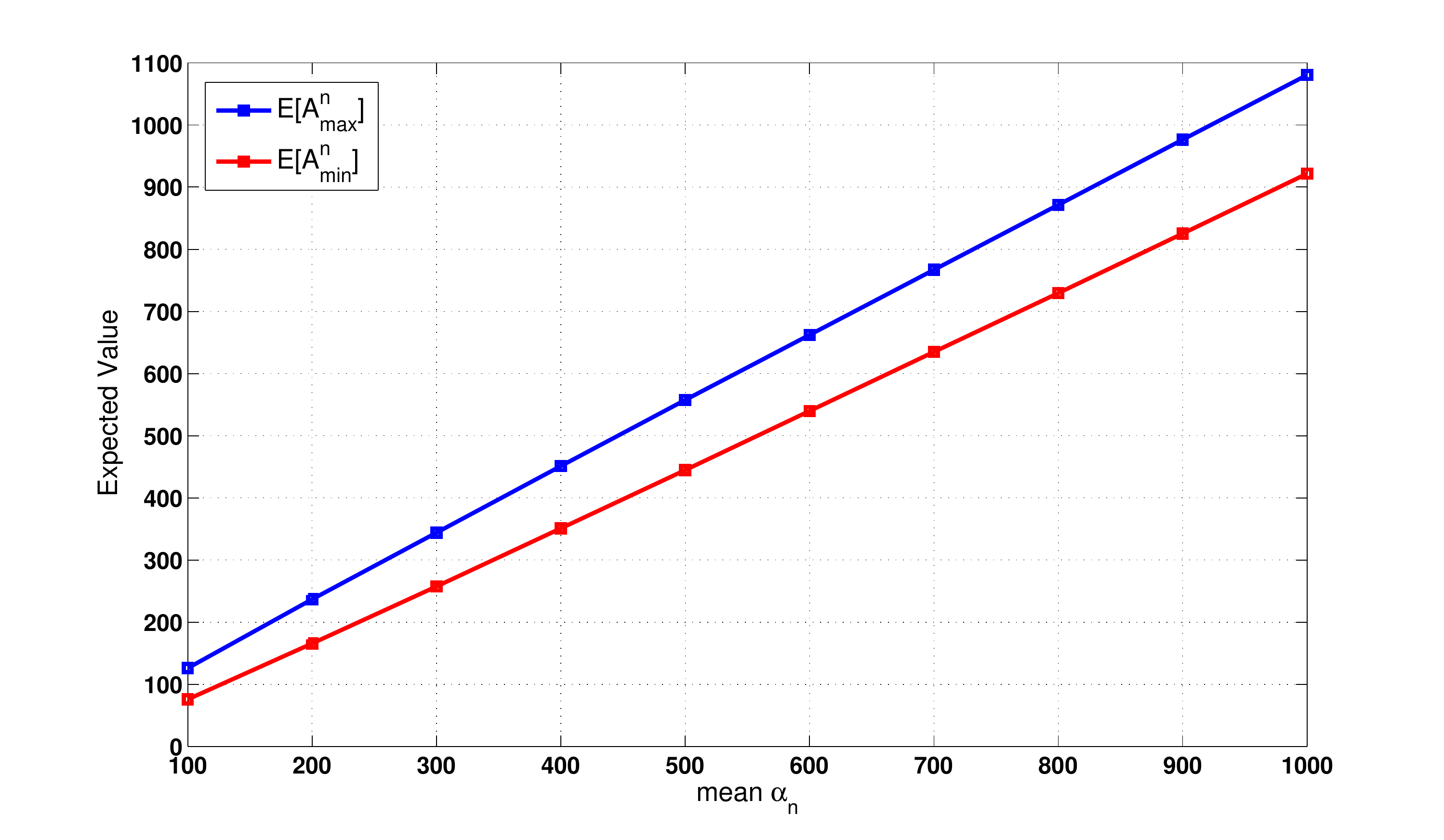}
			\caption{}
            \label{R}
		\end{subfigure}%
		\caption{(a) The values of $E[A_{max}^n]$ and $E[A_{min}^n]$ for various $\alpha_{n}$. (b) $E[A_{max}^n]/E[A_{min}^n]$ for various$\alpha_{n}$. $K=100$.}
	\end{flushleft}
\end{figure}

\section{Optimal Algorithms}
\label{optimal}

\subsection{peak-demand minimization - Modified Dijkstra}

In this section, we propose an optimal algorithm that can be used when $N$ is relatively small, to solve the peak-demand minimization problem. Prioir to describing the algorithm, we present a graph representation of the problem. Then, the algorithm is described, which is a modification of Dijkstra's Shortest Path (SP) algorithm. In addition, we provide a proof of correctness for the proposed algorithm, and analyze its runtime. Further, we discuss the algorithm's drawbacks, compared to the heuristics presented earlier.

\subsubsection{Graph Representation}
We represent the problem as a directed graph $G = (V,E)$: $E\rightarrow \mathbb{R}_+$, where $V$ is the set of vertices, and $E$ is the set of edges with positive weights. The graph representation is depicted in Fig. \ref{pdp_graph}. The vertices are divided into layers, where each layer consists of vertices labeled with all the possible combinations of $N-1$ threshold policies. In addition, a source and destination vertices are added with empty labels. The edges connect every two consecutive layers as the following. The connection is between each two vertices, where the last $N-2$ prices in the label of the first vertex equals the first $N-2$ prices in the label of the second vertex. More formally, every vertex in layer $i$, where $i< K-N+2$, with the label $\tau_{a},\tau_{b}, \dots ,\tau_{y} $ is connected to every vertex in layer $i+1$ with the label $\tau_{b},\dots, \tau_{y},\tau_{z}$, where $a,b,y,$ and $z \in \{1, \dots, N\}$. The weight of the edge is the total consumption at time period $i+N-1$, when the prices in the time interval $[i,i+N-1]$ is set to the threshold policy sequence $\tau_{a}, \tau_{b}, \dots, \tau_{y},\tau_{z}$. In addition, the source vertex is connected to all vertices in layer $1$, with a weight equal to the maximum consumption in the time interval $[1,N-1]$, when the prices in that time interval are set to the prices in the label of the adjacent vertex. Finally, the vertices in the last layer are connected to the destination vertex, with a weight that equals 0.

\begin{figure}
\resizebox {\columnwidth} {!} {%
\hspace*{0cm}\begin{tikzpicture}[ %
   ->, 
   >=stealth,
   shorten >=1pt,
   node distance=1.75cm,
   thick,
   state/.style={%
   fill=white,
   draw=black,
   text=black
  }   
  ]

    \node[state] (s) {$s$};
    
    \node[state] (DOT1) [right of=s,draw=none,node distance=3.5cm] {$\ldots$};
    \node[state] (B1) [above of=DOT1,node distance=2cm] {$\tau_{1},\tau_{1}, \dots,\tau_{2}$};
    \node[state] (A1) [above of=B1,node distance=2cm] {$\tau_{1},\tau_{1}, \dots ,\tau_{1} $};
    \node[state] (Note) [left=0cm, above of=A1,draw=none,node distance=1.5cm,] { Vertex label size= $N-1$};
    \draw [dashed,-] (A1) -- (Note);
    \node[state] (C1) [below of=DOT1,node distance=2cm] {$\tau_{N},\tau_{N}, \dots ,\tau_{N-1} $};
    \node[state] (D1) [below of=C1,node distance=2cm] {$\tau_{N},\tau_{N}, \dots ,\tau_{N} $};
    \node[state] (S1Note) [below of=D1,draw=none,node distance=1.5cm,] { Stage 1};

    \node[state] (DOT2) [right of=DOT1,draw=none,node distance=3.5cm] {$\ldots$};
    \node[state] (B2) [above of=DOT2,node distance=2cm] {$\tau_{1},\tau_{1}, \dots ,\tau_{2} $};
    \node[state] (A2) [above of=B2,node distance=2cm] {$\tau_{1},\tau_{1}, \dots ,\tau_{1} $};
    \node[state] (C2) [below of=DOT2,node distance=2cm] {$\tau_{N},\tau_{N}, \dots ,\tau_{N-1} $};
    \node[state] (D2) [below of=C2,node distance=2cm] {$\tau_{N},\tau_{N}, \dots ,\tau_{N} $};
    \node[state] (S2Note) [below of=D2,draw=none,node distance=1.5cm,] { Stage 2};
    
    \node[state] (DDOT3) [right of=DOT2,draw=none,node distance=2cm] {$\ldots$};
    \node[state] (DDOT2) [above of=DDOT3,draw=none,node distance=2cm] {$\ldots$};
    \node[state] (DDOT1) [above of=DDOT2,draw=none,node distance=2cm] {$\ldots$};
    \node[state] (DDOT4) [below of=DDOT3,draw=none,node distance=2cm] {$\ldots$};
    \node[state] (DDOT5) [below of=DDOT4,draw=none,node distance=2cm] {$\ldots$};

    \node[state] (DOT3) [right of= DDOT3, draw=none,node distance=2cm] {$\ldots$};
    \node[state] (B3) [above of=DOT3,node distance=2cm] {$\tau_{1},\tau_{1}, \dots ,\tau_{2} $};
    \node[state] (A3) [above of=B3,node distance=2cm] {$\tau_{1},\tau_{1}, \dots ,\tau_{1} $};
    \node[state] (C3) [below of=DOT3,node distance=2cm] {$\tau_{N},\tau_{N}, \dots ,\tau_{N-1} $};
    \node[state] (D3) [below of=C3,node distance=2cm] {$\tau_{N},\tau_{N}, \dots ,\tau_{N} $};
    \node[state] (S3Note) [below of=D3,draw=none,node distance=1.5cm,] {Stage $K-N+1$};
    
    \node[state] (DOT4) [right of=DOT3, draw=none,node distance=3.5cm] {$\ldots$};
    \node[state] (B4) [above of=DOT4,node distance=2cm] {$\tau_{1},\tau_{1}, \dots ,\tau_{2} $};
    \node[state] (A4) [above of=B4,node distance=2cm] {$\tau_{1},\tau_{1}, \dots ,\tau_{1} $};
    \node[state] (C4) [below of=DOT4,node distance=2cm] {$\tau_{N},\tau_{N}, \dots ,\tau_{N-1} $};
    \node[state] (D4) [below of=C4,node distance=2cm] {$\tau_{N},\tau_{N}, \dots ,\tau_{N} $};
    \node[state] (S4Note) [below of=D4,draw=none,node distance=1.5cm,] { Stage $K-N+2$};

    \node[state] (t) [right of=DOT4,node distance=3.5cm] {$t$};

\path (s)  edge  node [left=0.05, align=center ]  {$\max\limits_{1\leq k \leq N-1}u(k,\lambda)$}   (A1)
      (s)  edge  node [right=0.2, align=center]  {$\max\limits_{1\leq k \leq N-1}u(k,\lambda)$}  (B1)
      (s)  edge  node [right=0.35, align=center]  {$\max\limits_{1\leq k \leq N-1}u(k,\lambda)$}  (C1)
      (s)  edge  node [left=0.10, align=center]  {$\max\limits_{1\leq k \leq N-1}u(k,\lambda)$}  (D1)
      
      (A1)  edge  node [above=0.2, align=center]  {$u(N,\lambda)$}   (A2)
      (A1)  edge  node [right=0.2, align=center] {$u(N,\lambda)$}     (B2)
      (D1)  edge  node [left=0.2, align=center]  {$u(N,\lambda)$}   (C2)
      (D1)  edge  node [below=0.2, align=center]  {$u(N,\lambda)$}  (D2)
     
      (A3)  edge  node [above=0.2, align=center]  {$u(K,\lambda)$}    (A4)
      (A3)  edge  node [right=0.2, align=center] {$u(K,\lambda)$}   (B4)
      (D3)  edge  node [left=0.2, align=center]  {$u(K,\lambda)$}    (C4)
      (D3)  edge  node [below=0.2, align=center] {$u(K,\lambda)$}    (D4)
      
      (A4)  edge  node [right=0.1, align=center ]  {$0$}   (t)
      (B4)  edge  node [left=0.2, align=center]  {$0$}   (t)
      (C4)  edge  node [left=0.2, align=center]  {$0$}   (t)
      (D4)  edge  node [right=0.1, align=center]   {$0$}   (t)

;
\end{tikzpicture}
}
\caption{Peak demand pricing using a directed graph.}
\label{pdp_graph}
\end{figure}
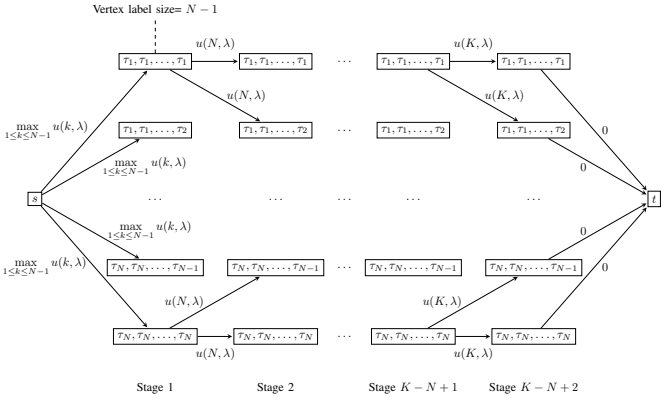

\subsubsection{Algorithm Description}
Using the graph model, the problem of peak-demand minimization is equivalent to the minimax path problem, which has been well studied in the literature \cite{dijkstra_proof},\cite{gluss1961minimax}, and can be solved by modifying the relaxation condition of Dijkstra's SP algorithm. The modification is as follows. Let $d(v)$ denote the cost of the vertex $v$, and $w(u,v)$ denote the weight of the edge connecting the vertices $u$ and $v$. The relaxation condition in Dijkstra's algorithm is,

\setlength{\interspacetitleruled}{-.4pt}%
\LinesNumberedHidden
\begin{algorithm}

    \If{ $d(v) > d(u) + w(u,v)$}
    {
        $d(v)$ = $d(u) + w(u,v)$
    }
\end{algorithm}

\begin{flushleft}
The modification for the relaxation condition is,
\end{flushleft}

\setlength{\interspacetitleruled}{-.4pt}%
\LinesNumberedHidden
\begin{algorithm}

    \If{ $d(v) > max (d(u), w(u,v))$}
    {
        $d(v) = max (d(u), w(u,v))$
    }
\end{algorithm}

After running Dijkstra's algorithm with the above modification, the solution will be stored in the destination vertex, which is a path from the source vertex to the destination vertex. The vertices along this path contains the sequence of prices that minimizes the peak demand, while excluding the repeated prices. The repeated prices are the first $N-2$ prices in the label of the vertices that belong to a stage $>$ 1. 

\begin{theorem}
The modified Dijkstra's algorithm returns the correct and optimal solution of the peak-demand minimization problem, with an expected runtime of $O(K \sum_{t=1}^{N}\sum_{s=0}^{N-t} \alpha_{s+t} + K \log K)$
\end{theorem}

\begin{proof}
First, we prove the correctness and the optimality of the algorithm as follows. The consumption at any time period $t$ can be completely determined by the prices chosen in the time interval $[t-N+1,t]$, since no consumer arriving at a time period $\leq t-N$ can contribute any demand in period $t$, as her deadline would expire earlier. This fact is the reason behind labeling the vertices with a sequence of $N-1$ prices. The combined prices of two adjacent vertices, defined as the prices in the label of the first vertex followed by the last price in the label of the second vertex, are in total $N$ prices that can determine the consumption of the time period that corresponds to the $N$-th price. This proves that any edge connecting stage $i$ to stage $i+1$ can be weighted with the correct consumption at time period $i+N-1$, given only the prices of the two adjacent vertices. In addition, recall that the edges connecting the source node to stage $1$ are weighted with the maximum consumption in the time interval $[1,N-1]$. This weighting is due to the fact that only the maximum consumption (peak demand) is what matters along the optimal path. Additionally, this weighting can be obviously determined given the $N-1$ prices in stage $1$. To show that a total of $K-N+2$ stages is needed in the graph, recall that the vertices in stage $1$ hold $N-1$ prices that correspond to the first $N-1$ time periods. The remaining $K-N+1$ time periods are represented by the remaining stages. Since each of these stages adds one price that corresponds to one time period, the total number of the remaining stages is $K-N+1$. Adding the first stage to these stages, the total number of stages in the graph is indeed $K-N+2$. This fact also justifies the 0 weighting of the edges between the last stage and the destination vertex, since no more consumption is to be determined.

Given the above proof, it is obvious that the paths in the graph between the source and destination vertices represent all the possible sequences of prices and their consequent consumption, in the time interval $[1,K]$. We are interested in the sequence of prices that minimizes the peak demand. In other words, the problem is to find a path from the source vertex to the destination vertex such that the maximum weight along the path is minimized. To show that the modified Dijkstra's algorithm solves this problem optimally, we refer to the work presented by A. Blum \cite{dijkstra_proof}, who proved that our described algorithm is optimal.

Next, we prove that the algorithm has an expected runtime of $O(K \sum_{t=1}^{N}\sum_{s=0}^{N-t} \alpha_{s+t} + K \log K)$ as follows. The total number of combinations of $N-1$ threshold policies is $N^{N-1}$, since there are $N$ threshold policies, and the size of each combination is $N-1$. Therefore, the algorithm creates $N^{N-1}$ vertices at each layer. The total number of layers is $K-N+2$ layers. In addition, the algorithm creates a source and destination vertices. Therefore, the total number of vertices to be created is $2+(N^{N-1})(K-N+2)$. In addition to the vertices, the algorithm connects the source vertex to every vertex in layer $1$, which is a total of $N^{N-1}$ vertices. Also, every vertex in layer $i$, where $i< K-N+2$, is connected to $N$ vertices in layer $i+1$, resulting in a total of $N^{N}$ total of edges for each layer $i$. Moreover, every vertex in layer $K-N+2$ is connected to the distention vertex, resulting in $N^{N-1}$ additional edges. Therefore, the total number of edges to be created is $2 N^{N-1}+ (N^{N})(K-N+1)$ edges. Furthermore, the weighting of each edge (while ignoring the edges to the destination vertex) requires the accumulation of demands in $N$ time periods that correspond to the prices in the adjacent vertices. We recall that the expected number of operations to accumulate the demands at each time period is no more than $\sum_{t=1}^{N}\sum_{s=0}^{N-t} \alpha_{s+t}$. Therefore, the expected number of operations to weight each edge is no more than $N\sum_{t=1}^{N}\sum_{s=0}^{N-t} \alpha_{s+t}$. Based on the above analysis, and assuming that $N$ is constant, it follows that the expected number of operations required for constructing the graph is $O(K \sum_{t=1}^{N}\sum_{s=0}^{N-t} \alpha_{s+t})$.

So far, we calculated the expected runtime required for constructing the graph. Once the graph is constructed, the modified Dijkstra's algorithm is run on the graph. According to \cite{fredman1987fibonacci}, the most efficient Dijkstra's algorithm on a directed graph with positive weight edges, is the one where the priority queue is a Fibonacci heap, and has a time complexity of $O(|E| + |V| \log |V|)$. By substituting our early calculations, and assuming that $N$ is constant, the time complexity of the modified Dijkstra's algorithm is $O(K \log K)$. Consequently, It follows that the overall expected runtime required for constructing the graph and running the modified Dijkstra's algorithm is  $O(K \sum_{t=1}^{N}\sum_{s=0}^{N-t} \alpha_{s+t} + K \log K)$.
\end{proof}


\subsubsection{Drawbacks}
One of the main drawbacks of the modified Dijkstra's algorithm is the high memory requirement. It can bee seen that the total number of required vertices to be stored is $2+(N^{N-1})(K-N+2)$. In addition, the total number of edges in the graph is $2 N^{N-1}+ (N^{N})(K-N+1)$. Since both of these values are exponential in $N$, we expect a fast growth in the memory requirement, when $N$ increases. TABLE \ref{sp_table} illustrates the memory requirement for a fixed time horizon $K=24$, and various values of $N$. The memory requirement reaches $4.76\times10^{18}$ total of vertices and edges, when $N=15$. We note that no such memory requirement is needed in the heuristics and the uniform pricing approach described earlier, since they do not construct a graph to solve the problem.

\begin {table}[t]
\setlength\extrarowheight{2pt}
\caption {Memory requirements for the modified Dijkstra ($K=24$). $|V|$ and $|E|$ are the number of vertices and edges.}
\label{sp_table}

\begin{center}
\begin{tabular}{ | C{1cm} | C{2cm} | C{2cm} | C{2cm} |  }
 \hline

  \textbf{N} & $|V|$ & $|E|$ & $|V|+|E|$ \\
  \hline
  \textbf{3}   &  209	 & 612 & 821  \\
  \hline
 \textbf{6}  & 155522 & 902016 &  1057538 \\
 \hline
 \textbf{9} 	& 7.32$\times10^{8}$ &	6.28$\times10^{9}$ & 7.02$\times10^{9}$  \\
 \hline
 \textbf{12} 	& 1.04$\times10^{13}$ &	1.17$\times10^{14}$ & 1.28$\times10^{14}$ \\
 \hline
 \textbf{15} 	& 3.21$\times10^{17}$ &	4.44$\times10^{18}$ & 4.76$\times10^{18}$ \\
 \hline

\end{tabular}
\end{center}
\end {table}

Another drawback of the modified Dijkstra's algorithm is related to its runtime. Although the expected runtime of the algorithm is $O(K \sum_{t=1}^{N}\sum_{s=0}^{N-t} \alpha_{s+t} + K \log K)$ (assuming $N$ is constant), the total number of operations performed by the algorithm is exponential in $N$. 
This causes the algorithm to become rapidly slow, with the increase of $N$. Unlike the modified Dijkstra's algorithm, the exponential term $W$ in the runtime of the Sliding-Window heuristic can be toned to avoid slowing down the heuristic. In addition, the Greedy heuristic and the uniform pricing approach do not perform exponential number of operations, based on our early analysis. Therefore, the methods presented earlier are obviously faster than the modified Dijkstra's algorithm.

It follows from the drawbacks discussed in this section that the modified Dijkstra's algorithm achieves the optimality on the expense of a high memory requirement and runtime. On the other hand, the heuristics presented earlier sacrifice optimality for the memory and speed. Consequently, we conclude that the modified Dijkstra's algorithm is suitable in practice, when the parameter $N$ is relatively small.

\subsection{Matching Demand to Supply - Dijkstra's Shortest Path}
In this section, we propose an optimal algorithm, which is similar to the modified Dijkstra's algorithm presented in the previous section, and can be used when $N$ is relatively small, to solve the problem of matching demand to supply. We initially present a modification for the graph representation presented in the previous section. Then, the algorithm is described, which is the standard Dijkstra's SP algorithm. In addition, we provide a proof of correctness for the proposed algorithm, and analyze its runtime and drawbacks.

\subsubsection{Graph Representation}
In this section, we present a modification for the directed graph presented in the previous section, Fig. \ref{pdp_graph}, to represent the problem of matching demand to supply. The modification is in the weighting of the edges as follows. Let $MSE_{[i,j]}$ denotes the MSE in the time interval $[i,j]$. An edge connecting a vertex in stage $i$ to a vertex in stage $i+1$ is weighted with $MSE_{[i+N-1,i+N-1]}$, when the prices in the time interval $[i,i+N-1]$ is set to the combined prices of the adjacent vertices. In addition, the source vertex is connected to all vertices in stage $1$, with a weight equals $MSE_{[1,N-1]}$, when the prices in the time interval $[1,N-1]$ is set to the prices in the label of the adjacent vertex. The rest of the graph representation remains unchanged.

\subsubsection{Algorithm Description}
Using the graph model, the problem of matching demand to supply can be solved using the standard Dijkstra's SP algorithm. After running Dijkstra's algorithm, the solution will be stored in the destination vertex, which is a path from the source vertex to the destination vertex. The vertices along this path contains the sequence of prices that matches demand to supply, while excluding the repeated prices. We recall that the repeated prices are the first $N-2$ prices in the label of the vertices that belong to a stage $>$ 1.

\begin{theorem}
Dijkstra's Shortest Path algorithm returns the correct and optimal solution of the problem of matching demand to supply.
\end{theorem}

\begin{proof}
Recall that the consumption at any time period $t$ can be completely determined by the prices chosen in the time interval $[t-N+1,t]$. Therefore, the combined prices of two adjacent vertices, are in total $N$ prices that can determine the consumption of the time period that corresponds to the $N$-th price. This proves that any edge connecting stage $i$ to stage $i+1$ can be weighted with the correct value of $MSE_{[i+N-1,i+N-1]}$, given only the prices of the two adjacent vertices. In addition, recall that the edges connecting the source vertex to stage $1$ are weighted with $MSE_{[1,N-1]}$. This weighting can be obviously determined given the $N-1$ prices in stage $1$. To show that a total of $K-N+2$ stages is needed, and justify the 0 weighting of the edges connected to the destination vertex, we refer to the proof of correctness provided in the previous section.

Given the above proof, it is obvious that the paths in the graph between the source and destination vertices represent all the possible sequences of prices and their consequent MSE, in the time interval $[1,K]$. We are interested in the sequence of prices that minimizes MSE. In other words, the problem is to find the shortest path from the source vertex to the destination vertex. To show that Dijkstra's SP algorithm solves this problem optimally, we refer to the work presented by A. Blum \cite{dijkstra_proof}, who proved that Dijkstra's algorithm is optimal.
\end{proof}

\subsubsection{Time Complexity \& Drawbacks}

The only differences between the modified Dijkstra's algorithm (previous section) and the Dijkstra's SP algorithm presented in this section are in the weighting of the edges and in the relaxation condition. For both algorithms, the weighting of the edges requires the same number of operations to accumulate the consumers' demand. Also, the number of operations required in the relaxation condition of the edges is constant for both algorithms. Therefore, the SP algorithm has the same expected runtime as the modified Dijkstra's algorithm. Also, it is obvious that both algorithms have the same memory requirements.

Recall that the drawbacks of the modified Dijkstra's algorithm are the high memory requirement and its runtime. Since the SP algorithm has the same expected runtime and memory requirements as the modified Dijkstra's algorithm, we conclude the same drawbacks for the SP algorithm. Consequently, the SP algorithm is only suitable when the parameter $N$ is relatively small. This claim is supported by the experimental results present in the next section.

\section{Experiments}
\label{experiments}
In this section, we present our experimental work for the proposed algorithms. We compare the approximation ratio of the proposed heuristics for the problems of peak-demand minimization and matching demand to supply. In addition, we compare the runtime of Dijkstra's algorithm with that of the proposed heuristics. Throughout the experiments, we vary the parameters $K$ and $N$ while recording the average of 30 experiments for each parameter value.

\subsection{Algorithms Implementation}

We implemented our heuristics using Java along with the CPLEX API to compute the heuristics approximation ratio. We also implemented Dijkstra's algorithm in Java, with a Fibonacci heap data structure as the priority queue. For some of the experiments (when $N$ is small), we used our implementation for Dijkstra's algorithm to verify the heuristics approximation ratio computed by CPLEX.

\subsection{Experimental Data}

We used real appliances data available on Pecan Street's Dataport website \cite{dataport}. The used data consists of the power demand of 64 appliances' jobs during 100 time periods. Using this data while randomly generating the deadline of the jobs, we evaluated the performance of the proposed algorithms in terms of approximation ratio and runtime. We now describe the experiments and results in detail.

\subsection{Heuristics Performance - Peak-Demand Minimization}
We conducted separate experiments for when the consumers' composition was both known and unknown a priori. We used a constant value of $N=3$. In each experiment, we considered three versions of the Sliding-Window heuristic. These versions use window sizes of $N$, $2N$, and $N^{2}$, respectively. We ran our simulations for time horizons in the range $K=[3,100]$. For each time horizon, we recorded the average approximation ratio of each heuristic. 

\subsubsection{Experiment 1: Known Consumer Composition}
We examined the heuristics' performances when the number of arrivals and their demands are known at all times. As mentioned earlier, we ran the experiment for various time horizons, and recorded the average approximation ratio of the heuristics.

\subsubsection{Experiment 2: Unknown Consumer Composition}
In this condition, we examined the performance of the heuristics when the number of future arrivals and their demands were unknown, meaning the number of arrivals and their demands were known only when the consumers arrive. We ran the Greedy heuristic and the uniform-pricing approach for the case when no estimation for the future arrivals and demands was made, as they optimize each time period without considering the future demands. Then, we used the estimator proposed by M. Ohannessian et al.~\cite{ohannessian2014dynamic} to estimate the future arrival and demand rates. These rates were then used by the Sliding-Window heuristic for future time periods.

\subsubsection{Experimental Results}
The results of Experiment 1 and 2 are illustrated in Fig. \ref{pdp_homo_offline} and \ref{pdp_homo_online}, correspondingly. For both experiments, the Sliding-Window heuristic outperformed the Greedy heuristic and the uniform pricing approach. In the first experiment, the Sliding-Window heuristic recorded an overall average approximation ratio of 1.13 when using a window size of $N$, 1.04 when using a window size of $2N$, and 1.01 when using a window size of $N^{2}$. In the second experiment, the Sliding-Window heuristic recorded an overall average approximation ratio of 1.14 when using a window size of $N$, 1.06 when using a window size of $2N$, and 1.02 when using a window size of $N^{2}$. On the other hand, the Greedy heuristic had an overall average approximation ratio of 1.2, and the uniform pricing approach had an overall average approximation ratio of 1.15, in both experiments. It follows that with the increase of the window size, the approximation ratio of the Sliding-Window heuristic is approaching 1.0. This is consistent with our early predictions that a larger window size means larger optimization coverage, which consequently leads to a lower approximation ratio. Nevertheless, with the heterogeneity of demands, the estimator performance drops as highlighted by M. Ohannessian et al. \cite{ohannessian2014dynamic}, which explains the slight increase in the approximation ratio of the Sliding-Window heuristic in the second experiment.

\begin{figure*}[t!]
	\begin{flushleft}
		\begin{subfigure}[t]{0.25\textwidth}
			\centering
			\includegraphics[height=1.4in, width=\linewidth]{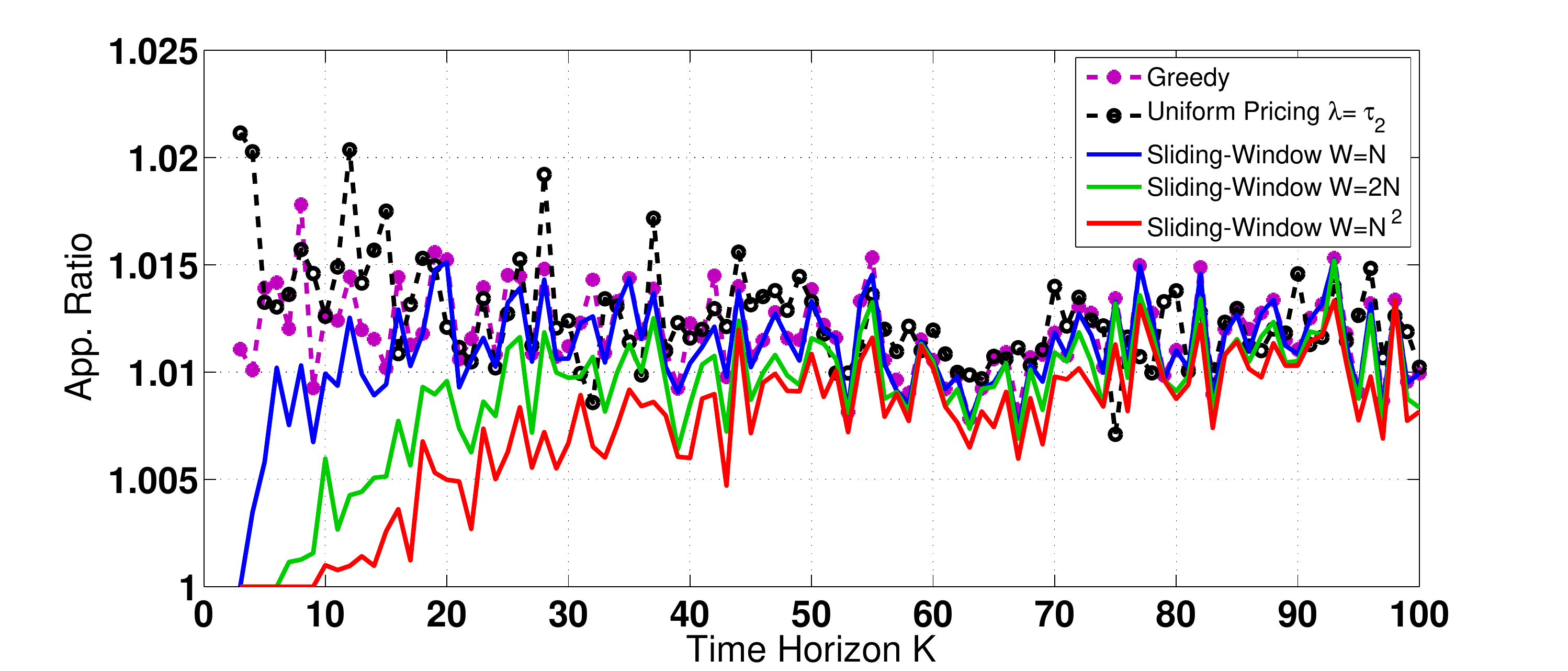}
			\caption{}
			\label{pdp_homo_offline}
		\end{subfigure}%
		~ 
		\begin{subfigure}[t]{0.25\textwidth}
			\centering
			\includegraphics[height=1.4in, width=\linewidth]{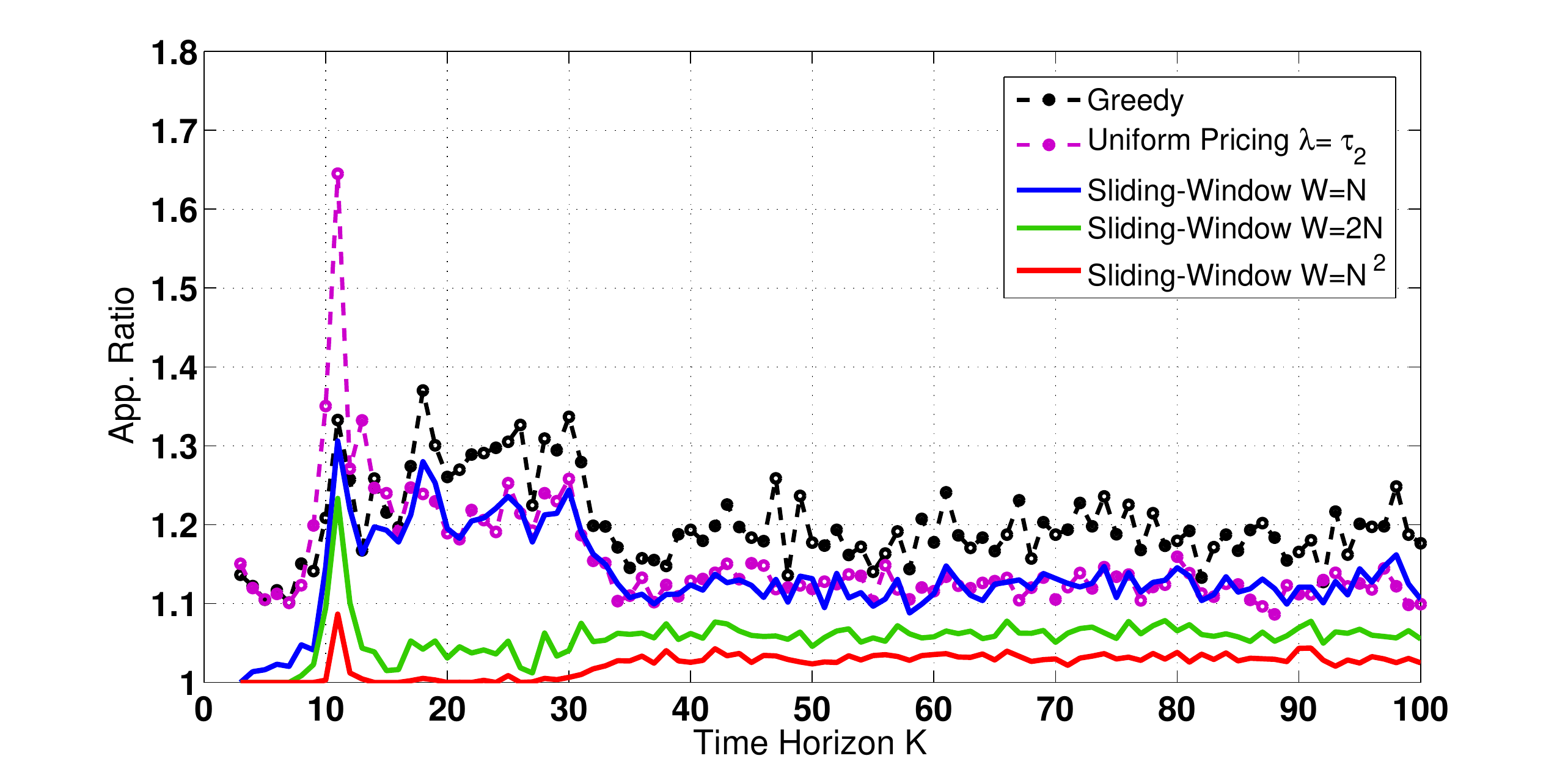}
			\caption{}
			\label{pdp_homo_online}
		\end{subfigure}%
		~
		\begin{subfigure}[t]{0.25\textwidth}
			\centering
			\includegraphics[height=1.4in, width=\linewidth]{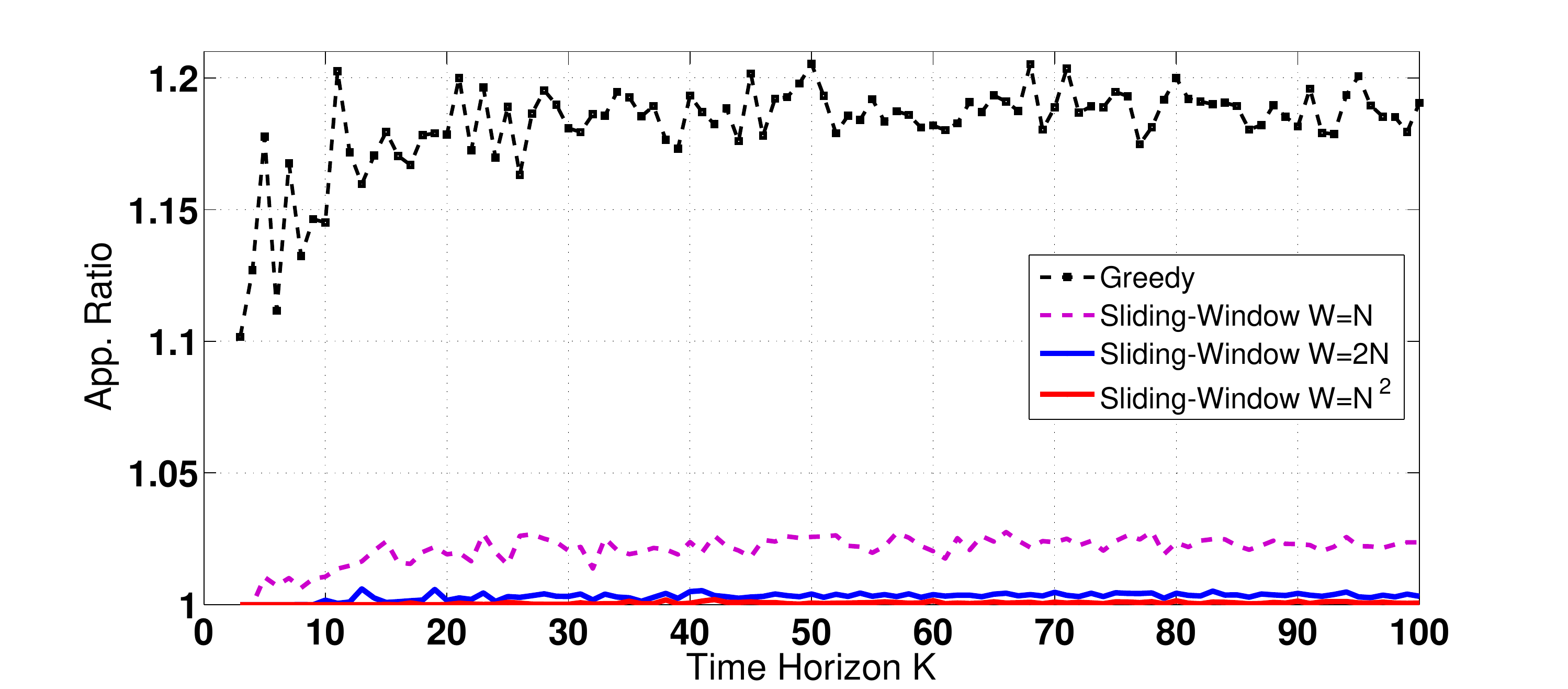}
			\caption{}
			\label{dsm_homo_offline}
		\end{subfigure}%
		~
		\begin{subfigure}[t]{0.25\textwidth}
			\centering
			\includegraphics[height=1.4in, width=\linewidth]{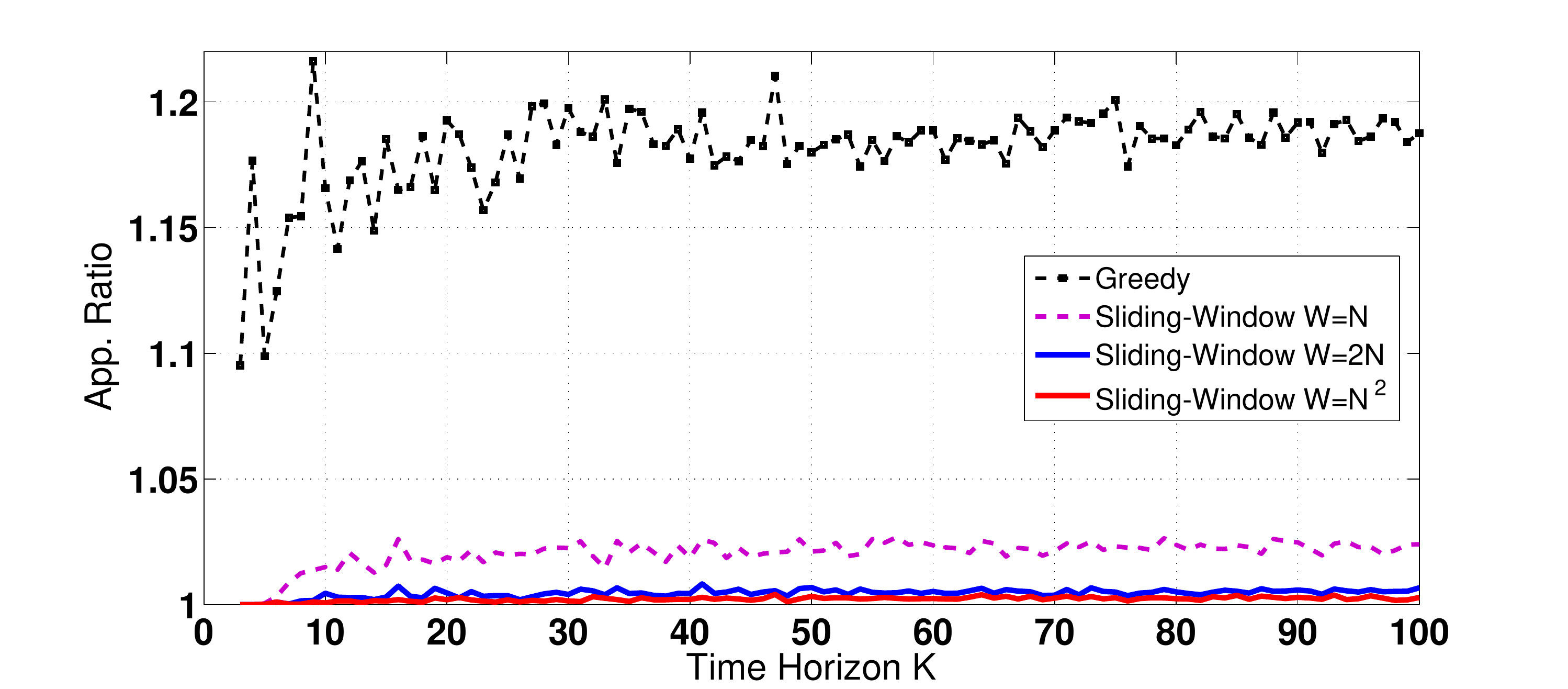}
			\caption{}
			\label{dsm_homo_online}
		\end{subfigure}
		
		\caption{Heuristic approximation ratios for peak-demand minimization given (a)~known {\em a priori} and (b) unknown consumer composition.  Heuristic approximation ratios for matching demand to supply given (c)~known {\em a priori} and (d) unknown consumer composition.   A sliding window was used to estimate arrival and demand rates given unknown consumer composition.}
	\end{flushleft}
\end{figure*}

\subsection{Heuristics Performance - Matching Demand to Supply}
In this section, we examine the performance of the the two generic heuristics proposed earlier, for solving the problem of matching demand to supply. As in the previous section, we conducted two different experiments based on whether the consumers' composition is known beforehand. Also, we used the same values for the parameters $K$,$N$, and $W$, and recorded the average approximation ratio of each heuristic.

The results of the two experiments are illustrated in Fig. \ref{dsm_homo_offline} and \ref{dsm_homo_online}. For both experiments, the Sliding-Window heuristic outperformed the Greedy heuristic. In the first experiment, the Sliding-Window heuristic recorded an overall average approximation ratio of 1.08 when using a window size of $N$, 1.01 when using a window size of $2N$, and 1.001 when using a window size of $N^{2}$. In the second experiment, the Sliding-Window heuristic recorded an overall average approximation ratio of 1.17 when using a window size of $N$, 1.19 when using a window size of $2N$, and 1.21 when using a window size of $N^{2}$. On the other hand, the Greedy heuristic had an overall average approximation ratio of 1.42 in both experiments. These results are consistent with the results of the previous section, as they show that with the increase of the window size, the approximation ratio of the Sliding-Window heuristic approaches 1.0. Yet, when the consumers' composition includes heterogeneous demands and is estimated, the estimator's performance drops.  This leads to an increase in the approximation ratio of the Sliding-Window heuristic.

\begin{figure*}[t]
\vspace{-.05in}
\begin{flushleft}
\begin{subfigure}[t]{0.25\textwidth}
\centering
\includegraphics[height=1.4in, width=\linewidth]{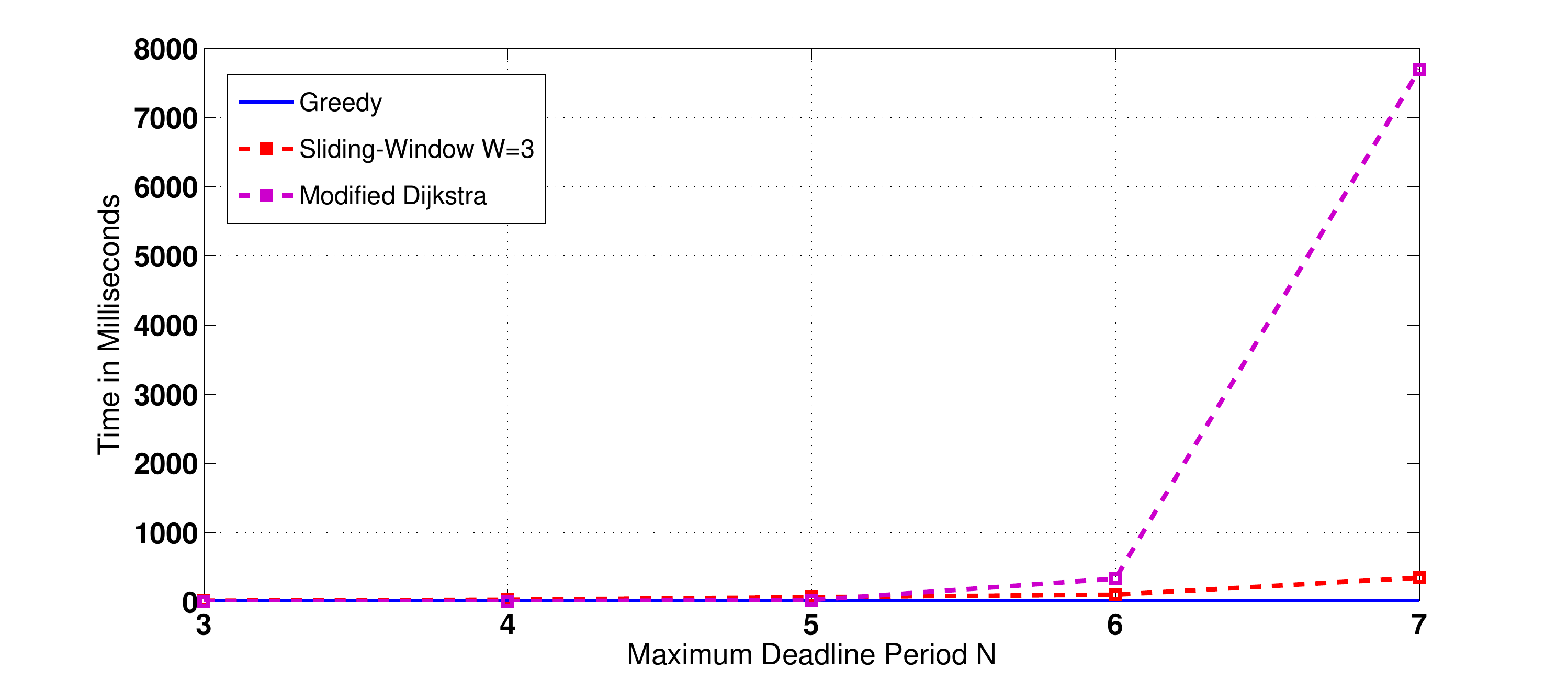}
\caption{}
\label{pdp_homo_time}
\end{subfigure}%
~ 
\begin{subfigure}[t]{0.25\textwidth}
\centering
\includegraphics[height=1.4in, width=\linewidth]{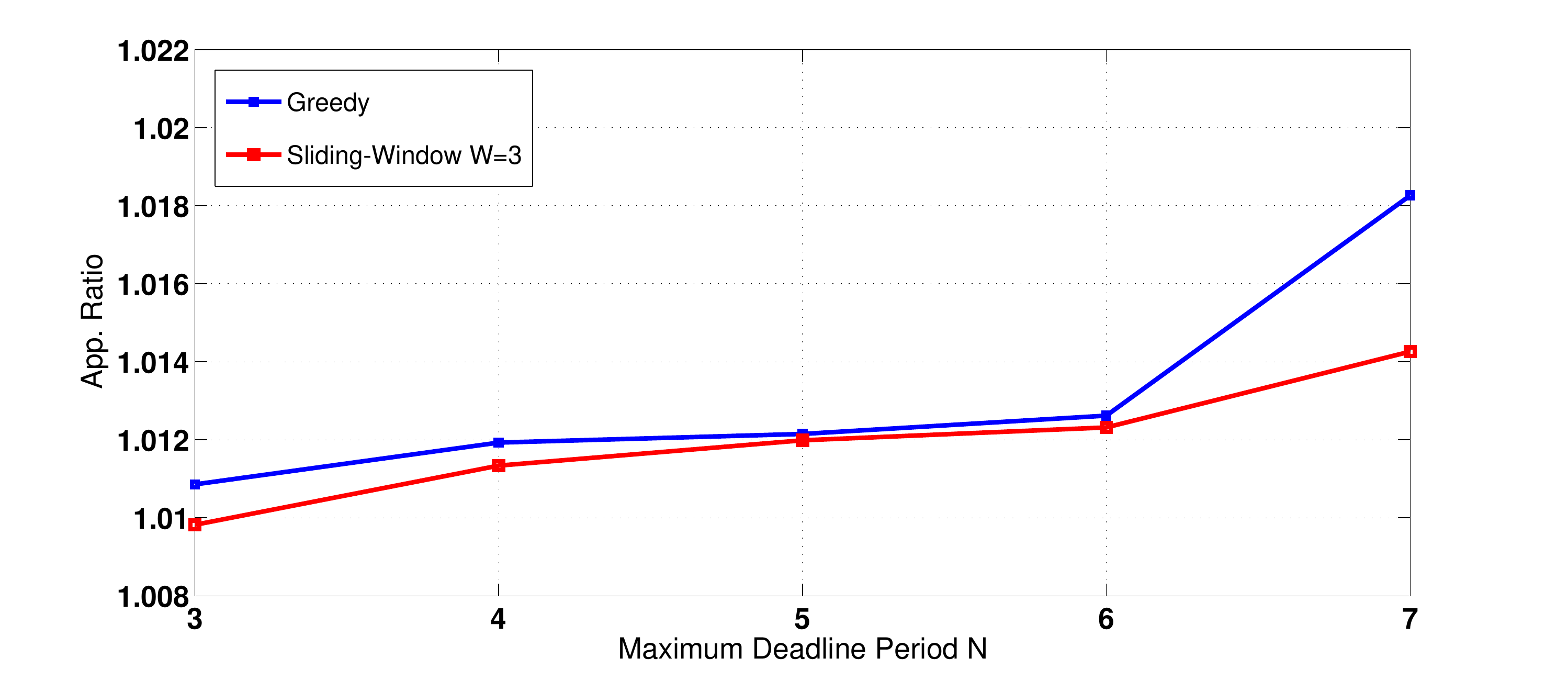}
\caption{}
\label{pdp_homo_ratio}
\end{subfigure}%
~
\begin{subfigure}[t]{0.25\textwidth}
\centering
\includegraphics[height=1.4in, width=\linewidth]{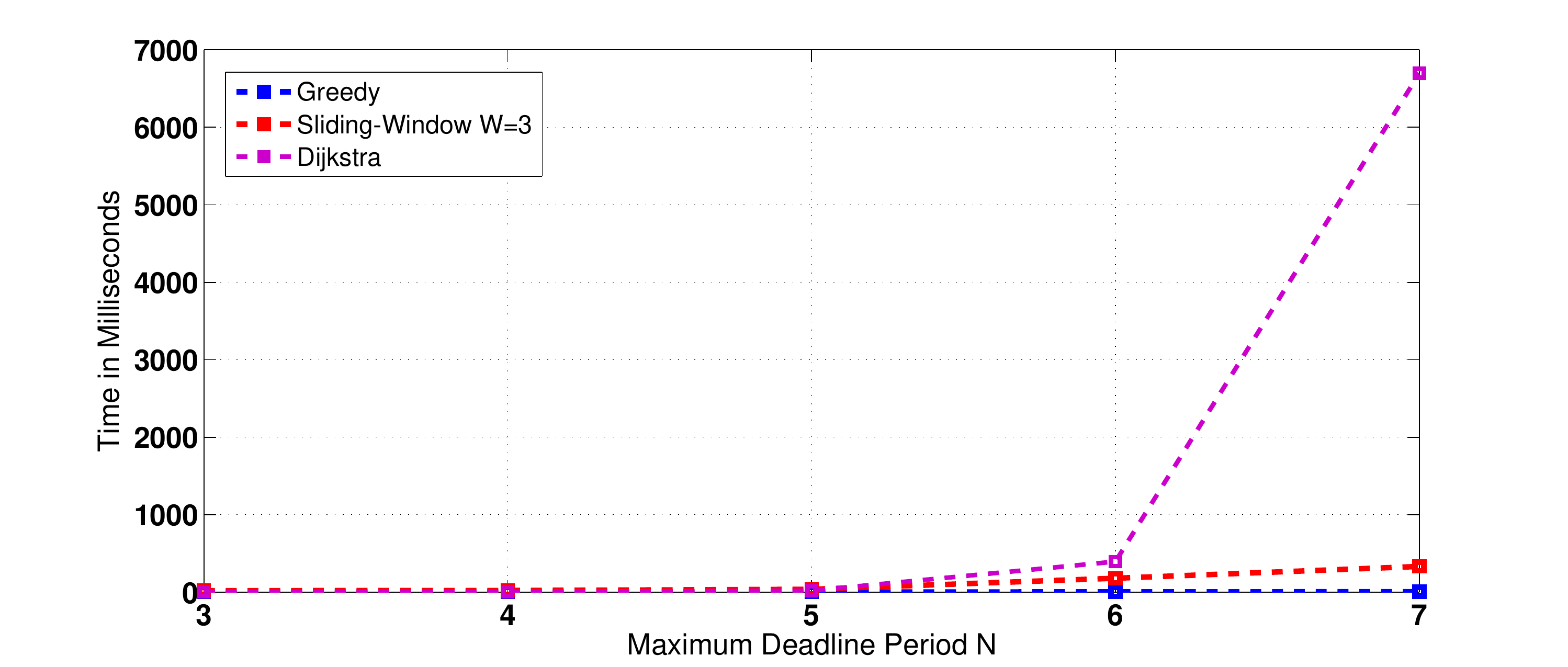}
\caption{}
\label{dsm_homo_time}
\end{subfigure}%
~
\begin{subfigure}[t]{0.25\textwidth}
\centering
\includegraphics[height=1.4in, width=\linewidth]{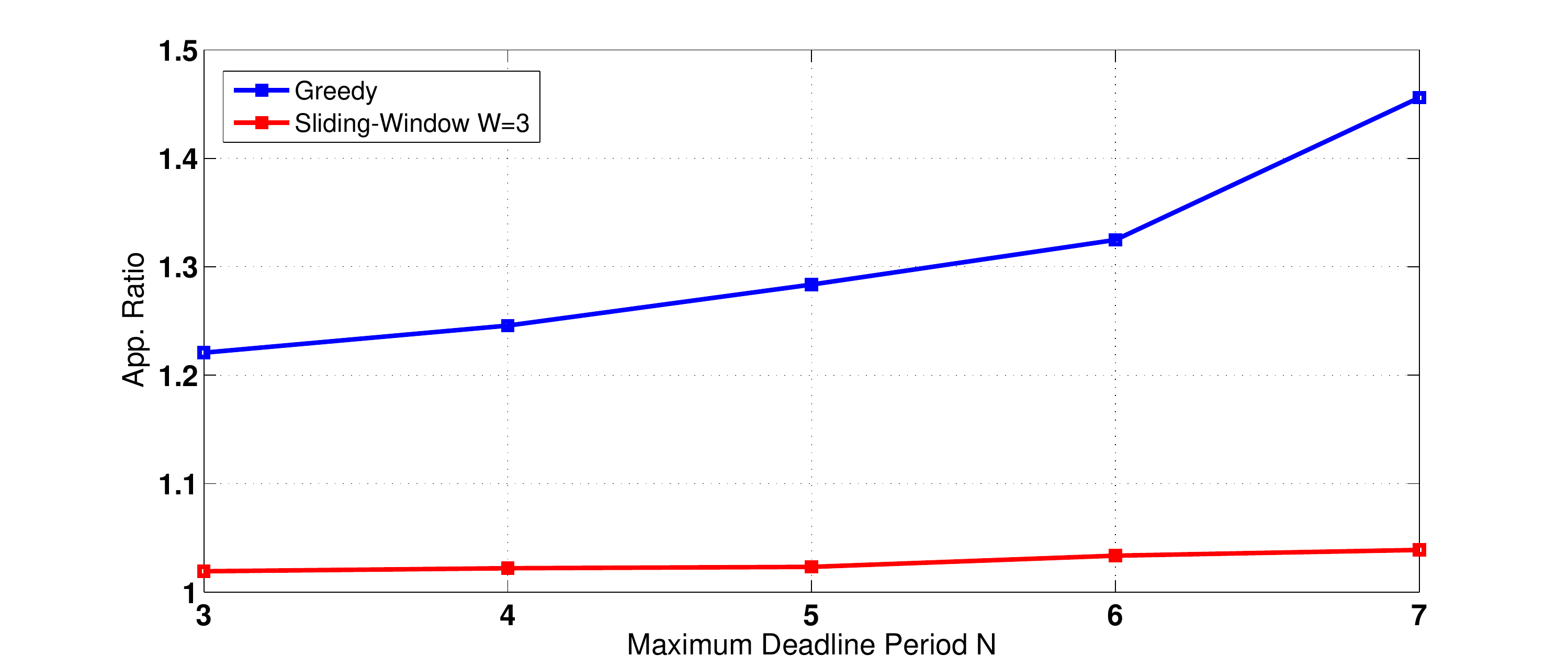}
\caption{}
\label{dsm_homo_ratio}
\end{subfigure}
\caption{(a) Runtime of the modified Dijkstra compared to the heuristics for peak-demand minimization. (b) Approximation ratio of the heuristics for peak-demand minimization. (c) Runtime of Dijkstra compared to the heuristics for matching demand to supply. (d) Approximation ratio of the heuristics for matching demand to supply. In all experiments, $K=24$.}
\end{flushleft}
\end{figure*}

\subsection{Runtime - Peak-Demand Minimization}

In this section, we examine the runtime of the modified Dijkstra's algorithm compared to the heuristics proposed earlier, for solving the problem of peak-demand minimization. In addition to the runtime of the heuristics, we keep track of the heuristics approximation ratio, to show the heuristics gaining of the speed on the expense of losing optimality. We conducted an experiment, where we recorded the average runtime of the algorithms and the average approximation ratio of the heuristics for various values of $N$ in the range $N=[3,7]$. Throughout the experiment, we fixed the value of $K$ to 24 time periods. In addition, we used the Sliding-Window heuristic with a window size $W=3$.

The experimental results are illustrated in Fig. \ref{pdp_homo_time} and \ref{pdp_homo_ratio}. Fig. \ref{pdp_homo_time} shows the fast growth in the runtime of the modified Dijkstra's algorithm, reaching an average of 7.9 seconds when $N=7$.  This is caused by the exponential term $N$ in the number of operations performed by the algorithm.  The Greedy and Sliding-Window heuristics had significantly lower runtimes than the modified Dijkstra's algorithm. The Sliding-Window heuristic had a higher overall average runtime than the Greedy heuristic, which was 109 milliseconds. At the same time, the Sliding-Window heuristic had a lower overall average approximation ratio than the Greedy heuristic, which was 1.2. The Greedy heuristic had an overall average runtime of 0.5 milliseconds, and an overall average approximation ratio of 1.28. These results are consistent with our early analysis that the heuristics sacrifice optimality for the speed. Additionally, based on the above results, the modified Dijkstra's algorithm is only suitable, when $N$ is relatively small.

\subsection{runtime - Matching Demand to Supply}
In this section, we examine the runtime of the  Dijkstra's SP algorithm compared to the heuristics proposed earlier for matching demand to supply.  As in the previous section, we measure the heuristics' approximation ratios, and use the same values of $K$, $N$, and $W$.

The experimental results are illustrated in Fig. \ref{dsm_homo_time} and \ref{dsm_homo_ratio}. Fig. \ref{dsm_homo_time} illustrates the fast growth in the runtime of the Dijkstra's SP algorithm, reaching an average of 6.8 seconds when $N=7$. As in the previous section, this result can be justified by the exponential term $N$ in the number of operations performed by the algorithm. On the other hand, the Greedy and Sliding-Window heuristics had a significantly lower runtime than the Dijkstra's SP algorithm. The Sliding-Window heuristic had an overall average runtime of 117 milliseconds, while the Greedy heuristic had an overall average runtime of 2 milliseconds. At the same time, the Sliding-Window heuristic had an overall average approximation ratio of 1.18, while the Greedy heuristic had an overall average approximation ratio of 1.75. These results show that the heuristics obviously sacrifice optimality for the speed, and the Dijkstra's SP algorithm is only suitable when $N$ is relatively small.

\section{Conclusion}
\label{conc}
In this paper, we studied the problems of peak-demand minimization, and matching demand to supply in the smart grid, using dynamic pricing. First, we proposed generic heuristics to minimize the peak load, and match demand to supply. In addition, we provided theoretical analysis for the uniform pricing approach in the context of peak-demand minimization. Our theoretical analysis provide a guarantee that in the typical configuration of the system, the approximation ratio of the uniform pricing approach on the average case is constant. Furthermore, we proposed an optimal algorithm for each of the dynamic pricing problems, which can be used when the maximum deadline period of the power jobs is relatively small. Our experimental results showed that the heuristics and the uniform pricing approach perform generally well, with an approximation ratio below 2 for all the experiments. Also, the experiments demonstrated an obvious trade-off between optimality and speed. Consequently, we conclude that the optimal algorithms are only suitable when the maximum deadline period of the power jobs is relatively small.


\begin{thebibliography}{10}
	
	\bibitem{al2014reducing}
	W.~Y.~Al Awadi.
	\newblock Reducing of the energy consumption in the federal buildings in uae
	using lighting and control technologies.
	\newblock 2014.
	
	\bibitem{hammad2010energy}
	F.~Hammad and B.~Abu-Hijleh.
	\newblock The energy savings potential of using dynamic external louvers in an
	office building.
	\newblock {\em Energy and Buildings}, 42(10):1888--1895, 2010.
	
	\bibitem{logenthiran2012demand}
	T.~Logenthiran, D.~Srinivasan, and T.~Z. Shun.
	\newblock Demand side management in smart grid using heuristic optimization.
	\newblock {\em Smart Grid, IEEE Transactions on}, 3(3):1244--1252, 2012.
	
	\bibitem{borgs2014optimal}
	C.~Borgs, O.~Candogan, J.~Chayes, I.~Lobel, and H.~Nazerzadeh.
	\newblock Optimal multiperiod pricing with service guarantees.
	\newblock {\em Management Science}, 60(7):1792--1811, 2014.
	
	\bibitem{roozbehani2012volatility}
	M.~Roozbehani, M.~A. Dahleh, and S.~K. Mitter.
	\newblock Volatility of power grids under real-time pricing.
	\newblock {\em Power Systems, IEEE Transactions on}, 27(4):1926--1940, 2012.
	
	\bibitem{ohannessian2014dynamic}
	M.~Ohannessian, M.~Roozbehani, D.~Materassi, M.~Dahleh, et~al.
	\newblock Dynamic estimation of the price-response of deadline-constrained
	electric loads under threshold policies.
	\newblock In {\em American Control Conference (ACC), 2014}, pages 2798--2803,
	2014.
	
	\bibitem{dataport}
	{Pecan Street's Dataport}.
	\newblock \url{https://dataport.pecanstreet.org/}.
	
	\bibitem{conejo2010real}
	A.~J. Conejo, J.~M. Morales, and L.~Baringo.
	\newblock Real-time demand response model.
	\newblock {\em Smart Grid, IEEE Transactions on}, 1(3):236--242, 2010.
	
	\bibitem{ahn2007pricing}
	H.~Ahn, M.~G{\"u}m{\"u}s, and P.~Kaminsky.
	\newblock Pricing and manufacturing decisions when demand is a function of
	prices in multiple periods.
	\newblock {\em Operations Research}, 55(6):1039--1057, 2007.
	
	\bibitem{ibars2010distributed}
	C.~Ibars, M.~Navarro, and L.~Giupponi.
	\newblock Distributed demand management in smart grid with a congestion game.
	\newblock In {\em Smart grid communications, IEEE International Conference on},
	pages 495--500, 2010.
	
	\bibitem{li2011optimal}
	N.~Li, L.~Chen, and S.~H. Low.
	\newblock Optimal demand response based on utility maximization in power
	networks.
	\newblock In {\em IEEE Power and Energy Society General Meeting}, pages 1--8,
	2011.
	
	\bibitem{ramchurn2012putting}
	S.~D. Ramchurn, P.~Vytelingum, A.~Rogers, and N.~R. Jennings.
	\newblock Putting the'smarts' into the smart grid: a grand challenge for
	artificial intelligence.
	\newblock {\em Communications of the ACM}, 55(4):86--97, 2012.
	
	\bibitem{strbac2008demand}
	G.~Strbac.
	\newblock Demand side management: Benefits and challenges.
	\newblock {\em Energy policy}, 36(12):4419--4426, 2008.
	
	\bibitem{wijaya2013matching}
	T.~K. Wijaya, K.~M. Larson, and K.~Aberer.
	\newblock Matching demand with supply in the smart grid using agent-based
	multiunit auction.
	\newblock In {\em Communication Systems and Networks, IEEE International
		Conference on}, pages 1--6, 2013.
	
	\bibitem{taylor2013consolidated}
	J.~A. Taylor, A.~Nayyar, D.~S. Callaway, and K.~Poolla.
	\newblock Consolidated dynamic pricing of power system regulation.
	\newblock {\em Power Systems, IEEE Transactions on}, 28(4):4692--4700, 2013.
	
	\bibitem{berger1989real}
	A.~W. Berger and F.~C. Schweppe.
	\newblock Real time pricing to assist in load frequency control.
	\newblock {\em Power Systems, IEEE Transactions on}, 4(3):920--926, 1989.
	
	\bibitem{mathMax}
	{Stack Exchange}.
	\newblock
	\url{http://math.stackexchange.com/questions/868299/expected-maximum-of-a-sequence-of-i-i-d-poissons}.
	
	\bibitem{mathMin}
	{Stack Exchange}.
	\newblock
	\url{http://math.stackexchange.com/questions/713414/minimum-of-identical-independent-poisson-random-variables}.
	
	\bibitem{dijkstra_proof}
	A.~Blum.
	\newblock {\em Graph Algorithms II}.
	\newblock Carnegie Mellon University, 2011.
	
	\bibitem{gluss1961minimax}
	B.~Gluss.
	\newblock The minimax path in a search for a circle in a plane.
	\newblock {\em Naval Research Logistics Quarterly}, 8(4):357--360, 1961.
	
	\bibitem{fredman1987fibonacci}
	Michael~L Fredman and Robert~Endre Tarjan.
	\newblock Fibonacci heaps and their uses in improved network optimization
	algorithms.
	\newblock {\em Journal of the ACM (JACM)}, 34(3):596--615, 1987.
	
\end{thebibliography}

\appendix
\section{NP-Completeness}
\label{np}

\begin{description}
	
	\item[\textsc{MSE Minimization }]
	
	\item[Given: ] The number of periods $m\in\mathbb{N}^+$, $n$ consumers with demand $d^i(k)$ for $i\in[n],k\in[m[$; threshold policies $\tau_1, \ldots,\tau_n$; supplies $S(k)\in \mathbb{N}^+, k\in[m]$; and a number $\alpha>0$.
	
	\item[Question: ] Is there a price vector $(\lambda_1,\lambda_2,\ldots,\lambda_K)$ such that $\Omega=\frac{1}{K}\sum_{k\in[m]}(u(k,\lambda_k)-S_k)^2$ greater than or equal to $t$?  
\end{description}

\begin{theorem}
	The \textsc{MSE minimization} problem is NP-complete.
\end{theorem}
\begin{proof}
	The problem  \textsc{MSE minimization} is in NP since given a price vector  $\vec\lambda$, one can check in polynomial time if $\Omega\ge \alpha$. To prove the hardness part, we use a reduction from \textsc{Subset-Sum} problem: given a set of positive integer numbers $\mathcal{S}=(a_1,a_2,\ldots,a_K)$, and a positive integer $B$, the question is of whether or not there is a subset $V$ of $\{1,\ldots,K\}$ such that $\sum_{i\in V}a_i=B$?
	
	Let $(\mathcal{S},B)$ be an instance of \textsc{Subset-Sum}. We assume without loss of generality that $a_1\ge a_2\ge \cdots\ge a_K$. We construct a corresponding instance $\mathcal{I}$ of the \textsc{MSE Minimization} problem with $K$ consumers $\{c_1,c_2,\ldots,c_K\}$ and $K+1$ periods. Consumer $c_i$ arrives at period $i$ and departures at period $K+1$,  for all $i\in[K]$. The demand of the consumer $c_i$ at the period $k$ is $d^i(k)=a_i$ iff $k=i$, and $0$  if $j\not=i$, for all $i\in[K]$ and $k\in[K+1]$. Since the consumers' time-to-go are different from each other, their threshold policies are different as well. Define the threshold policy of the consumer $c_i$ as $\tau_i=i$ for all $i\in[K]$\footnote{Note that here we can define the threshold policies $\tau_1,\ldots,\tau_K$ in an arbitrarily way as long as $\tau_1<\ldots<\tau_K$.}. Note that $\tau_1<\tau_2<\cdots<\tau_K$. The targeted power supply of the period $k$, denoted by $S(k)$, is $a_k/2$ for $k\in[K]$, and the supply of the last period is $S(K+1)=B$. Let $\vec\lambda=(\lambda_1,\ldots,\lambda_{K+1})$ be the optimal price vector that minimizes $\Omega=\frac{1}{K+1}\sum_{k\in[K+1]}(u(k,\lambda_k)-S(k))^2$, where $u(k,\lambda_k)$ denotes the total power consumption  at the period $k$, given the price $\lambda_k$. Let $\alpha=\frac{1}{K+1}\sum_{k=1}^Ka^2_k/4$. We call $\mathcal{I}$ a yes-instance of \textsc{Consumption Minimization} if and only if there is a price vector $\vec \lambda$ such that the value of $\Omega$ is at most $ \alpha$.
	
	It is not hard to see that it does not matter how the price at the period $k, k\in[K]$, is placed, the 
	squared difference between the power consumed $u(k,\lambda_k)$ and the supply $S(k)$ at this period is at least  $a_k^2/4$. Indeed, let's consider an arbitrary period $k\in[K]$. The claim is obviously clear  if $c_k$ does consume her power at this period.  Now assume that  $c_k$  shifts the demand $a_k$ to the next periods. The first case when no one consumes power at this period, the different between the supply and demand is still $a_k/2$. In the second case,  there is some consumer, say $c_{k'}$, $k'<k$, moves her demand $a_{k'}$  to the period $k$, it follows that $u(k,\lambda_k)-S(k)=a_{k'}-a_k/2\ge a_k-a_k/2=a_k/2$ as $a_{k'}\ge a_{k}$ for any $k>k'$. 
	
	
	Now, suppose  that $(\mathcal{S},B)$ be a yes-instance of \textsc{Subset-Sum}. We can place a price vector  $\vec\lambda$, that matches the demand and supply at the last period, i.e., $u(K+1,\lambda_{K+1})=B$, while ensures that all the consumers will either consume their power right at the first period they arrive, or shift their demand to the last period. In fact, the square error $\Omega$ will be shown to be equals to $\frac{1}{K+1}\sum_{k=1}^Ka^2_k/4$. Indeed, let $V$ be  a subset of $\{1,\ldots,K\}$ such that $\sum_{i\in V}a_i=B$. We place a price for each of periods (from the first to the last one) using the following rule: for $k\in[K]$, we set $\lambda_k>\tau_k$ if $k\in V$, and set $\tau_{k-1}<\lambda_k<\tau_k$ (here we define $\tau_0=0$), otherwise; and for the last period $K+1$, we set $\lambda_{K+1}<\min\{\tau_{k}|\,k\in V\}$. This price setting guarantees that every consumer $c_i$ will consume the power $a_i$ right at the period $i$ if $i\in V$, and shift the demand to the last period, otherwise, for all $i\in [K]$. Furthermore, those consumers, who did not consume the power at their first period, will do it only at the last period, where the price provided is lower than their threshold policies. It follows that $u(K+1,\lambda_{K+1})=B$. Hence, we have $\Omega=\frac{1}{K+1}\sum_{k=1}^Ka^2_k/4=\alpha$. 
	
	Conversely, if $(\mathcal{S},B)$ be a no-instance of \textsc{Subset-Sum}, there is no subset of $\mathcal{S}$ for which the sum of its elements is exactly $B$. Hence, for any price vector $\vec\lambda$, any shifting of the consumers' demands from the first $K$ periods to the last period $K+1$ will not get rid of the different between the total power consumption and the targeted power supply at the last period. Moreover, as we argued earlier, $(u(k,\lambda_k)-S(k))^2$ is at least $a^2_k/4$, for every period $k\in [K]$. Therefore, it must hold that $\Omega>\frac{1}{K+1}\sum_{k=1}^Ka^2_k/4>\alpha$.
	
	In conclusion, $(\mathcal{S},B)$ is a yes-instance of \textsc{Subset-Sum} if and only if the corresponding instance $\mathcal{I}$ is a yes-instance of \textsc{MSE Minimization}. This completes the proof.\end{proof}

\end{document}